\documentclass{amsart}

\usepackage{amsmath}
\usepackage{amsthm}
\usepackage{amsfonts}
\usepackage{dsfont}
\usepackage{enumitem}
\usepackage{color}

\newcommand{\reals}{\mathbb{R}}

\newcommand{\complex}{\mathbb{C}}

\newcommand{\integers}{\mathbb{Z}}

\newcommand{\paraa}[1]{\big(#1\big)}
\newcommand{\parab}[1]{\Big(#1\Big)}

\newcommand{\spacearound}[1]{\quad#1\quad}
\newcommand{\equivalent}{\spacearound{\Leftrightarrow}}
\renewcommand{\implies}{\spacearound{\Rightarrow}}

\newtheorem{theorem}{Theorem}[section]
\newtheorem{corollary}[theorem]{Corollary}
\newtheorem{lemma}[theorem]{Lemma}
\newtheorem{proposition}[theorem]{Proposition}

\theoremstyle{definition}
\newtheorem{definition}[theorem]{Definition}
\theoremstyle{remark}
\newtheorem{remark}[theorem]{Remark}
\numberwithin{equation}{section}

\renewcommand{\mid}{\mathds{1}}
\newcommand{\A}{\mathcal{A}}

\renewcommand{\d}{\partial}
\renewcommand{\P}{\mathcal{P}}
\newcommand{\nablab}{\bar{\nabla}}
\newcommand{\nablat}{\tilde{\nabla}}
\newcommand{\Der}{\operatorname{Der}}
\newcommand{\nablasub}[1]{\nabla_{\!#1}}
\newcommand{\nablatsub}[1]{\nablat_{\!#1}}

\newcommand{\nablad}{\nablasub{d}}

\newcommand{\g}{\mathfrak{g}}
\newcommand{\gphi}{\g_{\varphi}}
\newcommand{\Mphi}{M_{\varphi}}

\newcommand{\Z}{\mathcal{Z}}
\newcommand{\ZA}{\Z(\A)}
\newcommand{\half}{\frac{1}{2}}

\newcommand{\Sthreet}{S^3_\theta}
\newcommand{\Sthreetloc}{S^3_{\theta,\textup{loc}}}
\newcommand{\Ws}{W^\ast}
\newcommand{\Zs}{Z^\ast}
\newcommand{\qb}{\bar{q}}

\newcommand{\Wsq}{|W|^2}
\newcommand{\Zsq}{|Z|^2}
\newcommand{\xv}{\vec{x}}

\newcommand{\hhat}{\hat{h}}

\newcommand{\R}[3]{R(\d_{#1},\d_{#2})E_{#3}}

\newcommand{\Ttheta}{T^2_\theta}
\newcommand{\nablah}{\hat{\nabla}}

\title[Riemannian curvature of the noncommutative 3-sphere]{Riemannian curvature of the\\noncommutative 3-sphere}
\author{Joakim Arnlind and Mitsuru Wilson}

\address[Joakim Arnlind]{Department of Mathematics\\
Link\"oping University\\
581 83 Link\"oping\\
Sweden}
\email{joakim.arnlind@liu.se}

\address[Mitsuru Wilson]{%
Middlesex College\\
University of Western Ontario\\
London, Ontario\\
N6A 5B7\\
Canada}
\email{mwils57@uwo.ca}

\subjclass[2000]{}

\begin{document}

\maketitle

\begin{abstract}
  In order to investigate to what extent the calculus of classical
  (pseudo-)Riemannian manifolds can be extended to a noncommutative
  setting, we introduce pseudo-Riemannian calculi of modules over
  noncommutative algebras. In this framework, it is possible to prove
  an analogue of Levi-Civita's theorem, stating that there exists at
  most one torsion-free and metric connection for a given (metric)
  module, satisfying the requirements of a real metric
  calculus. Furthermore, the corresponding curvature operator has the
  same symmetry properties as the classical Riemannian curvature. As
  our main motivating example, we consider a pseudo-Riemannian
  calculus over the noncommutative 3-sphere and explicitly determine
  the torsion-free and metric connection, as well as the curvature
  operator together with its scalar curvature.
\end{abstract}

\section{Introduction}

\noindent The topological aspects of noncommutative geometry have been
extensively developed over the last decades, and there is a fine
understanding of how geometrical concepts generalize (or not) to the
noncommutative setting. Moreover, via spectral triples and Dirac
operators, metric aspects have also been thoroughly studied
\cite{c:ncgbook}. In particular, a noncommutative connection and curvature formalism is
worked out by A. Connes in \cite{c:cstaralgebre}, and now there is
also an understanding of scalar curvature in terms of heat kernel
expansions for spectral triples.

In recent years, several authors have made progress in computing the
scalar curvature for noncommutative tori, defined as a particular term
in the asymptotic heat kernel expansion, in analogy with classical
Riemannian geometry
\cite{cm:modularCurvature,fk:curvatureFourTori,fk:scalarCurvature,lm:modularCurvature}.
The novelty of this formulation is that it can be used to prove index
type theorems, such as a noncommutative Gauss-Bonnet theorem, which
can be shown to hold for conformal perturbations of the flat metric
for the noncommutative torus \cite{fk:gaussBonnet,ct:gaussBonnet}. The
computations rely on pseudo-differential calculus and are highly
technical and analytical in nature.  These results certainly have a
profound impact in the field of noncommutative geometry. An
interesting question is the existence of a curvature tensor, whose
scalar curvature coincides with the one arising from the heat kernel
expansion.

In an interesting paper by J. Rosenberg \cite{r:leviCivita}, an
algebraic approach to curvature of the noncommutative torus is taken,
much in the spirit of noncommutative differential geometry
\cite{c:cstaralgebre}. It turns out to be possible to construct a
Levi-Civita connection for certain classes of noncommutative tori,
whose curvature tensor and scalar curvature can easily be computed.

In this paper, we try to understand some of the prerequisites for
introducing traditional Riemannian geometry over a noncommutative
algebra, and formalize these ideas in \emph{pseudo-Riemannian calculi}
for which several classical results hold; in particular, there
exists at most one torsion-free and metric connection. Furthermore,
under certain hermiticity assumptions, the curvature tensor has all
the symmetries one finds in the differential geometric
setting. Although our framework is admittedly quite restrictive, and
only a few noncommutative manifolds fulfill the requirements, we
believe that our results contribute to the understanding of
noncommutative Riemannian geometry, by studying particular and well
known examples: the noncommutative torus and, foremost, the
noncommutative 3-sphere \cite{m:ncspheres}. In these examples, there
are natural choices of modules, corresponding to (sections of) tangent
bundles, which presents themselves when considering the manifolds as
embedded in Euclidean space (see
\cite{a:curvatureGeometric,ah:pseudoRiemannian,ahh:multilinear,abhhs:noncommutative}
for similar approaches making use of embeddings in Euclidean
space). Let us also point out that there are several other related
approaches to Riemannian structures in noncommutative geometry; see
e.g. \cite{cff:gravityncgeometry,dvmmm:onCurvature,ac:ncgravitysolutions,bm:starCompatibleConnections}.

For our main example, the noncommutative 3-sphere, we find it
interesting that our computations seem to introduce a type of
noncommutative local tangent bundle, in the following sense (as
discussed in Section \ref{sec:global.vs.local}). In general, the
module of vector fields is not a free module, which impedes to work
with directly from a computational point of view. Therefore, one
usually considers the manifold chart by chart and carry out pointwise
calculations. In noncommutative geometry, points are generally not
accessible, but the fact that the tangent bundle is locally free is
useful. However, since the objects we work with are intrinsically
global, the restriction of vector fields to a chart in the
noncommutative setting is not immediate.  Instead, we extend a local
basis of vector fields to global vector fields and consider the set of
vector fields in the local chart as a free submodule of the tangent
bundle, generated by the globalized vector fields. From a classical
point of view, computations may equally well be done with these vector
fields, keeping in mind that results can only be trusted for points
that belong to the given chart. Furthermore, we introduce an Ore
localization of the noncommutative 3-sphere, which is in direct
analogy with the algebra of functions in the chart provided by the
classical Hopf coordinates.

The paper is organized as follows: In Section \ref{sec:preliminaries}
we introduce a few basic concepts of noncommutative algebra that will
be used throughout the paper, in order to fix our notation. Section
\ref{sec:pseudo.riemannian.caculi} introduces pseudo-Riemannian
calculi, which provides a computational framework for Riemannian
geometry over noncommutative algebras. In Section
\ref{sec:curvature.of.calculus} the symmetry properties of the
curvature of a pseudo-Riemannian calculus is discussed, as well as the
possibility of introducing a scalar curvature. Section
\ref{sec:nc.torus} presents the noncommutative torus in the framework
of pseudo-Riemannian calculi, and we show that a unique (flat)
torsion-free and metric connection exists. In Section
\ref{sec:nc.three.sphere} we introduce the main motivating example for
this paper, the noncommutative 3-sphere. A real pseudo-Riemannian
calculus is constructed, including the unique torsion-free and metric
connection and, furthermore, the scalar curvature is
computed. Finally, in Section \ref{sec:global.vs.local}, we discuss
aspects of noncommutative localization in the context of the noncommutative
3-sphere.

\section{Preliminaries}
\label{sec:preliminaries}

\noindent In this section we shall recall the definitions of a few
basic algebraic objects, in order to set the notation for the rest of
the paper.  In the following, $\A$ will denote a unital $\ast$-algebra (over
$\complex$) with center $\ZA$. The set of derivations of $\A$ (into
$\A$) is denoted by $\Der(\A)$, and
for any derivation $d\in\Der(\A)$, there is a hermitian
conjugate $d^\ast$, given by $d^\ast(a) = \paraa{d(a^\ast)}^\ast$; a
derivation is called hermitian if $d^\ast = d$.

In this paper we shall mainly be concerned with right $\A$-modules. In
particular, the free (right) $\A$-module $(\A)^n$ has a canonical
basis given by $\{e_1,\ldots,e_n\}$ where
\begin{align*}
  e_i = (0,\ldots,0,\mid,0,\ldots,0)
\end{align*}
with the only nonzero element in the $i$'th position. An element
$U\in(\A)^n$ can be written as $U=e_iU^i$ (with an implicit sum over
$i$ from $1$ to $n$) for some (uniquely determined) elements
$U^1,\ldots,U^n\in\A$.

\begin{definition}
  Let $M$ be a right $\A$-module. A map $h:M\times M\to\A$ is called
  a hermitian form on $M$ if
  \begin{align*}
    &h(U,V+W) = h(U,V) + h(U,W)\\
    &h(U,Va) = h(U,V)a\\
    &h(U,V)^\ast = h(V,U).
  \end{align*}
  A hermitian form is \emph{non-degenerate} if $h(U,V)=0$ for all
  $V\in M$ implies that $U=0$. For brevity, we simply refer to a
  non-degenerate hermitian form as a \emph{metric} on $M$. The pair $(M,h)$,
  where $M$ is a right $\A$-module and $h$ is a hermitian form on $M$,
  is called a \emph{(right) hermitian $\A$-module}. If $h$ is a
  metric, we say that $(M,h)$ is a \emph{(right) metric $\A$-module}. 
\end{definition}

\noindent
Let us introduce affine connections on a right $\A$-module, adjusted
to fit the particular setting of this paper.

\begin{definition}
  Let $M$ be a right $\A$-module and let $\g\subseteq\Der(\A)$ be a
  (real) Lie algebra of hermitian derivations. An affine connection on $(M,\g)$ is a
  map $\nabla:\g\times M\to M$ such that
  \begin{enumerate}
  \item $\nablasub{d}(U+V) = \nablasub{d}U+\nablasub{d}V$,
  \item $\nablasub{\lambda d+d'}U = \lambda\nablasub{d}U+\nablasub{d'}U$,
  \item $\nablasub{d}(Ua) = \paraa{\nablasub{d}U}a+Ud(a)$,
  \end{enumerate}
  for all $U,V\in M$, $d,d'\in\g$, $a\in\A$ and $\lambda\in\reals$.
\end{definition}

\begin{remark}
  Note that since we are considering affine connections with respect
  to a \emph{subset} of $\Der(\A)$, it does not make sense in general
  to demand that $\nablasub{cd}U=c\nablasub{d}U$ for (hermitian) $c\in\ZA$, since
  $\g$ may not be closed under the action of $\ZA$. However, in the
  examples we consider it is true that $\nablasub{cd}U=c\nablasub{d}U$
  whenever $cd\in\g$. (In fact, this is a general statement which
  follows from Kozul's formula \eqref{eq:kozul.formula} as soon as
  $\varphi$, in Definition \ref{def:real.metric.calculus}, is linear
  over $\ZA$ in the above sense.)
\end{remark}

\section{Pseudo-Riemannian calculi}
\label{sec:pseudo.riemannian.caculi}

\noindent
In differential geometry, every derivation of $C^\infty(M)$ gives rise
to a (unique) vector field on the manifold $M$. Hence, in the
algebraic definition of a connection, where $\nablasub{d}U$ is defined
for $d\in\Der(C^\infty(M))$ and $U\in TM$, one may swap the two
arguments due to the fact that there is a one-to-one correspondence
between derivations and vector fields. For instance, this makes the
classical definition of torsion meaningful:
\begin{align*}
  T(U,V) = \nabla_UV-\nabla_VU-[U,V],
\end{align*}
from an algebraic point of view.  In a derivation based differential
calculus over a noncommutative algebra (see
e.g. \cite{dv:calculDifferentiel}), the arguments of a connection is
not on equal footing, partly due to the fact that the set of
derivations is in general not a module over the algebra. Thus, there
is no natural way to associate an element of the module to an
arbitrary derivation.

In this paper, we shall investigate the consequences of introducing a
correspondence, which assigns a unique element of a module to every
derivation in a Lie algebra $\mathfrak{g}\subseteq\Der(\A)$. This idea is
formalized in the following definition.

\begin{definition}\label{def:real.metric.calculus}
  Let $(M,h)$ be a (right) metric $\A$-module, let
  $\g\subseteq\Der(\A)$ be a (real) Lie algebra of hermitian
  derivations and let $\varphi:\g\to M$ be a $\reals$-linear
  map. If we denote the pair $(\g,\varphi)$ by $\gphi$, the
  triple $(M,h,\gphi)$ is called a \emph{real metric calculus} if
  \begin{enumerate}
  \item the image $\Mphi=\varphi(\g)$ generates $M$ as an $\A$-module,
  \item
    $h(E,E')^\ast=h(E,E')$
    for all $E,E'\in\Mphi$.
  \end{enumerate}
\end{definition}

\noindent 
The condition that the elements in the image of $\varphi$ have
hermitian inner products, corresponds to the fact that the metric is
real, and that the inner product of two real vector fields, is again a
real function. An important consequence of this assumption is that $h$
is truly symmetric on the image of $\varphi$,
i.e. $h(E,E')=h(E',E)$ for all $E,E'\in\Mphi$; a fact that
will repeatedly be used in the sequel.

In this setting, we shall introduce a connection on a real metric
calculus, and demand that the connection preserve the hermiticity of
$\Mphi$.

\begin{definition}
  Let $(M,h,\gphi)$ be a real metric calculus and let $\nabla$ denote
  an affine connection on $(M,\g)$. If
  \begin{align*}
    h(\nabla_dE,E') = h(\nabla_dE,E')^\ast
  \end{align*}
  for all $E,E'\in\Mphi$ and $d\in \g$ then $(M,h,\gphi,\nabla)$ is
  called a \emph{real connection calculus}.
\end{definition}

\noindent For a real connection calculus it is straightforward to
introduce the concept of a metric and torsion-free connection.

\begin{definition}
  Let $(M,h,\gphi,\nabla)$ be a real connection calculus over $M$. The
  calculus is \emph{metric} if
  \begin{align*}
    d\paraa{h(U,V)} = h\paraa{\nablasub{d}U,V} + h\paraa{U,\nablad V}
  \end{align*}
  for all $d\in\g$, $U,V\in M$, and \emph{torsion-free} if 
  \begin{align*}
    \nablasub{d_1}\varphi(d_2)-\nablasub{d_2}\varphi(d_1)
    -\varphi\paraa{[d_1,d_2]} = 0
  \end{align*}
  for all $d_1,d_2\in \g$. A metric and torsion-free real connection
  calculus over $M$ is called a \emph{pseudo-Riemannian calculus over $M$}. 
\end{definition}

\noindent Pseudo-Riemannian calculi will be the main objects of
interest to us, and they provide a framework in which one may carry
out computations in close analogy with classical Riemannian geometry.

The Levi-Civita theorem states that there is a unique torsion-free and
metric connection on the tangent bundle of a Riemannian manifold. In
the current setting, one can not guarantee the existence, but given a
real metric calculus, there exists at most one connection which is
both metric and torsion-free.

\begin{theorem}
  Let $(M,h,\gphi)$ be a real metric calculus over $M$. Then there exists
  at most one affine connection $\nabla$ on $(M,\g)$, such that
  $(M,h,\gphi,\nabla)$ is a pseudo-Riemannian calculus. 
\end{theorem}

\begin{proof}
  Assume that there exist two connections $\nabla$ and $\nablat$ such
  that $(M,h,\gphi,\nabla)$ and $(M,h,\gphi,\nablat)$ are pseudo-Riemannian
  calculi. Let us define 
  \begin{align*}
    \alpha(d,U) = \nablatsub{d}U-\nablasub{d}U,
  \end{align*}
  from which it follows that
  \begin{align*}
    \alpha(d,Ua) = \paraa{\nablatsub{d}U}a+Uda
    -\paraa{\nablasub{d}U}a-Uda = \alpha(d,U)a,
  \end{align*}
  for $a\in\A$, as well as
  \begin{align*}
    &\alpha(d_1+\lambda d_2,U) = \alpha(d_1,U)+\lambda\alpha(d_2,U)\\
    &\alpha(d,U+V) = \alpha(d,U) + \alpha(d,V),
  \end{align*}
  for $\lambda\in\reals$. By subtracting the conditions that
  $\nabla,\nablat$ are metric one obtains
  \begin{align}\label{eq:alpha.metric}
    h\paraa{\alpha(d,U),V} = -h\paraa{U,\alpha(d,V)},
  \end{align}
  and the torsion-free condition implies
  \begin{align}\label{eq:alpha.torsion}
    \alpha\paraa{d_1,\varphi(d_2)}
    =\alpha\paraa{d_2,\varphi(d_1)}
  \end{align}
  for $d_1,d_2\in\g$. Finally, requiring that 
  $h(\nablasub{d_1}\varphi(d_2),\varphi(d_3))$ and
  $h(\nablatsub{d_1}\varphi(d_2),\varphi(d_3))$ are hermitian gives
  \begin{align}\label{alpha.hermitian}
    h\paraa{\alpha(d_1,\varphi(d_2)),\varphi(d_3)}^\ast
    =h\paraa{\alpha(d_1,\varphi(d_2)),\varphi(d_3)}.
  \end{align}
  Now, let us make use of \eqref{eq:alpha.metric} and
  \eqref{eq:alpha.torsion} to compute (where $E_a=\varphi(d_a)$)
  \begin{align*}
    h\paraa{\alpha(d_1,E_2),E_3}
    &=h\paraa{\alpha(d_2,E_1),E_3}
    =-h\paraa{E_1,\alpha(d_2,E_3)}
    =-h\paraa{E_1,\alpha(d_3,E_2)}\\
    &=h\paraa{\alpha(d_3,E_1),E_2}
    =h\paraa{\alpha(d_1,E_3),E_2}
    =-h\paraa{E_3,\alpha(d_1,E_2)},
  \end{align*}
  which shows that
  \begin{align*}
    h\paraa{\alpha(d_1,\varphi(d_2)),\varphi(d_3)}^\ast
    =-h\paraa{\alpha(d_1,\varphi(d_2)),\varphi(d_3)}.
  \end{align*}
  Combining this result with \eqref{alpha.hermitian} yields
  \begin{align*}
    h\paraa{\alpha(d_1,\varphi(d_2)),\varphi(d_3)}=0,
  \end{align*}
  for all $d_1,d_2,d_3\in\g$. Since the image of $\varphi$ generates
  $M$ and $h$ is non-degenerate, it follows that $\alpha(d,U)=0$ for
  all $U\in M$ and $d\in\g$, which shows that
  \begin{align*}
    \nablatsub{d}U = \nablasub{d}U
  \end{align*}
  for all $d\in\g$ and $U\in M$.
\end{proof}

\noindent The Levi-Civita connection can be explicitly constructed
with the help of Kozul's formula, which gives the connection as
expressed in terms of the metric tensor. For pseudo-Riemannian
calculi, there is a corresponding statement.  

\begin{proposition}\label{prop:kozul.formula}
  Let $(M,h,\gphi,\nabla)$ be a pseudo-Riemannian calculus and assume that
  $d_1,d_2,d_3\in\g$. Then it holds that
  \begin{equation}\label{eq:kozul.formula}
    \begin{split}
      2h&(\nablasub{d_1}E_2,E_3)=
      d_1h(E_2,E_3)+d_2h(E_1,E_3)-d_3h(E_1,E_2)\\
      &\quad
      -h\paraa{E_1,\varphi([d_2,d_3])}
      +h\paraa{E_2,\varphi([d_3,d_1])}
      +h\paraa{E_3,\varphi([d_1,d_2])},
    \end{split}
  \end{equation}
  where $E_a=\varphi(d_a)$ for $a\in\{1,2,3\}$.
\end{proposition}

\begin{proof}
  Using the fact that $\nabla$ is a metric connection gives
  \begin{align}
      &d_1h(E_2,E_3) =
      h\paraa{\nablasub{d_1}E_2,E_3}+h\paraa{E_2,\nablasub{d_1}E_3}\label{eq:mc.ott}\\
      &d_2h(E_3,E_1) =
      h\paraa{\nablasub{d_2}E_3,E_1}+h\paraa{E_3,\nablasub{d_2}E_1}\label{eq:mc.tto}\\
      &d_3h(E_1,E_2) = 
      h\paraa{\nablasub{d_3}E_1,E_2}+h\paraa{E_1,\nablasub{d_3}E_2}\label{eq:mc.tot},
  \end{align}
  and since the connection is torsion-free one obtains
  \begin{align*}
    &h(E_3,\nablasub{d_2}E_1) =
    h(E_3,\nablasub{d_1}E_2)+h\paraa{E_3,\varphi([d_2,d_1])}\\
    &h(\nablasub{d_3}E_1,E_2) =
    h(\nablasub{d_1}E_3,E_2)+h\paraa{\varphi([d_3,d_1]),E_2}\\
    &h(E_1,\nablasub{d_3}E_2) = 
    h(E_1,\nablasub{d_2}E_3) + h\paraa{E_1,\varphi([d_3,d_2])}.
  \end{align*}
  Moreover, using that the connection is real enables us to rewrite
  the above equations in the following form
  \begin{align}
    &h(E_3,\nablasub{d_2}E_1) =
    h(\nablasub{d_1}E_2,E_3)+h\paraa{E_3,\varphi([d_2,d_1])}\label{eq:tf.tto}\\
    &h(\nablasub{d_3}E_1,E_2) =
    h(E_2,\nablasub{d_1}E_3)+h\paraa{\varphi([d_3,d_1]),E_2}\label{eq:tf.tot}\\
    &h(E_1,\nablasub{d_3}E_2) = 
    h(\nablasub{d_2}E_3,E_1) + h\paraa{E_1,\varphi([d_3,d_2])}.\label{eq:tf.ott}    
  \end{align}
  Inserting \eqref{eq:tf.tto} in \eqref{eq:mc.tto} and
  \eqref{eq:tf.tot},~\eqref{eq:tf.ott} in \eqref{eq:mc.tot} gives
  (together with \eqref{eq:mc.ott})
  \begin{align*}
    &h(\nablasub{d_1}E_2,E_3) =
    d_1h(E_2,E_3)-h(E_2,\nablasub{d_1}E_3)\\
    &h(\nablasub{d_1}E_2,E_3) = 
    d_2h(E_3,E_1)-h(\nablasub{d_2}E_3,E_1)-h\paraa{E_3,\varphi([d_2,d_1])}\\
    &0 =
    -d_3h(E_1,E_2)+h(E_2,\nablasub{d_1}E_3)+h\paraa{\varphi([d_3,d_1]),E_2}\\
    &\qquad +h(\nablasub{d_2}E_3,E_1) + h\paraa{E_1,\varphi([d_3,d_2])},
  \end{align*}
  and summing these three equations yields
  \begin{align*}
    2h(\nablasub{d_1}E_2,E_3)&=
    d_1h(E_2,E_3)+d_2h(E_3,E_1)-d_3h(E_1,E_2)\\
    &\quad
    -h\paraa{E_3,\varphi([d_2,d_1])}+h\paraa{\varphi([d_3,d_1]),E_2}
    + h\paraa{E_1,\varphi([d_3,d_2])},
  \end{align*}
  which proves \eqref{eq:kozul.formula}.
\end{proof}

\begin{remark}
  Note that Proposition \ref{prop:kozul.formula} gives an independent
  proof of the fact that the connection is unique, since the hermitian
  form $h$ is assumed to be nondegenerate.
\end{remark}

\noindent Now, let us show that the converse of Proposition
\eqref{eq:kozul.formula} is true; i.e. a connection satisfying
\eqref{eq:kozul.formula} gives a pseudo-Riemannian calculus.

\begin{proposition}\label{prop:Kozul.implies.metric.tf}
  Let $(M,h,\gphi)$ be a real metric calculus, and let $\nabla$ be an
  affine connection on $(M,\g)$ such that Kozul's formula (\ref{eq:kozul.formula})
  holds. Then $(M,h,\gphi,\nabla)$ is a pseudo-Riemannian calculus.
\end{proposition}

\begin{proof}
  From (\ref{eq:kozul.formula}) it follows immediately that
  $h(\nablasub{d_1}\varphi(d_2),\varphi(d_3))$ is hermitian since
  every term in the right hand side is hermitian, due to the fact that
  $(M,h,\gphi)$ is assumed to be a real metric calculus, which implies
  that $(M,h,\gphi,\nabla)$ is a real connection calculus. Next, let
  us show that the connection is metric.

  Let $d_1,d_2,d_3\in\g$ and set $E_a=\varphi(d_a)$. Using eq.
  (\ref{eq:kozul.formula}) twice (together with the fact that
  $(M,h,\gphi)$ is a real metric calculus), gives
  \begin{align*}
    h(\nablasub{d_1}E_2,E_3)+h(E_2,\nablasub{d_1}E_3) = d_1h(E_2,E_3).
  \end{align*}
  Since $\Mphi$ generates $M$, one may find
  $\{E_a=\varphi(d_a)\}_{a=1}^N$ such that one can write $U=E_aU^a$
  for all $U\in M$. It then follows that
  \begin{align*}
    &h(\nablasub{d}U,V) + h(U,\nablasub{d}V) \\
    &= h\paraa{(\nablasub{d}E_a)U^a+E_adU^a,E_bV^b}
    +h\paraa{E_aU^a,(\nablasub{d}E_b)V^b+E_bdV^b}\\
    &=(U^a)^\ast\paraa{h(\nablasub{d}E_a,E_b)+h(E_a,\nablasub{d}E_b)}V^b
    +d(U^a)^\ast h(E_a,E_b)V^b+(U^a)^\ast h(E_a,E_b)dV^b\\
    &= (U^a)^\ast dh(E_a,E_b)V^b
    +d(U^a)^\ast h(E_a,E_b)V^b+(U^a)^\ast h(E_a,E_b)dV^b\\
    &=d\paraa{(U^a)^\ast h(E_a,E_b)V^b}=dh(U,V),
  \end{align*}
  which shows that the affine connection is metric. Finally, let us show that the
  connection is torsion-free. For $d_1,d_2,d_3\in g$, with
  $E_a=\varphi(d_a)$, let us consider
  \begin{align*}
    T=h\paraa{\nablasub{d_1}E_2-\nablasub{d_2}E_1-\varphi([d_1,d_2]),E_3}.
  \end{align*}
  By using formula \eqref{eq:kozul.formula} for the first two
  terms, one obtains
  \begin{align*}
    T &= h\paraa{E_3,\varphi([d_1,d_2])}-h\paraa{\varphi([d_1,d_2]),E_3} = 0.
  \end{align*}
  Since the image of $\varphi$ generates $M$ one can conclude that 
  \begin{align*}
    h\paraa{\nablasub{d_1}E_2-\nablasub{d_2}E_1-\varphi([d_1,d_2]),U}=0
  \end{align*}
  for all $U\in M$, which implies that
  \begin{align*}
    \nablasub{d_1}E_2-\nablasub{d_2}E_1-\varphi([d_1,d_2]) = 0,
  \end{align*}
  since $h$ is nondegenerate.
\end{proof}

\noindent 
In particular examples, it is possible to
use Kozul's formula to construct a metric and torsion-free
connection. One of the cases, which is relevant to our examples, is when
$M$ is a free module.

\begin{corollary}\label{cor:connection.from.Kozul}
  Let $(M,h,\gphi)$ be a real metric calculus and let
  $\{\d_1,\ldots,\d_n\}$ be a basis of $\g$ such that
  $\{E_a=\varphi(\d_a)\}_{a=1}^n$ is a basis for $M$. If there exist
  $U_{ab}\in M$ such that
  \begin{equation}\label{eq:Kozul.Uab}
    \begin{split}
      2h&(U_{ab},E_c)=
      \d_ah(E_b,E_c)+\d_bh(E_a,E_c)-\d_ch(E_a,E_b)\\
      &\quad
      -h\paraa{E_a,\varphi([\d_b,\d_c])}
      +h\paraa{E_b,\varphi([\d_c,\d_a])}
      +h\paraa{E_c,\varphi([\d_a,\d_b])}
    \end{split}
  \end{equation}
  for $a,b,c=1,\ldots,n$, then there exists a connection $\nabla$,
  given by $\nablasub{\d_a}E_b = U_{ab}$, such that
  $(M,h,\gphi,\nabla)$ is a pseudo-Riemannian calculus.
\end{corollary}

\begin{proof}
  Assuming that such elements $U_{ab}\in M$ exist, define
  \begin{align*}
    \nablasub{\d_a}E_b = U_{ab}
  \end{align*}
  and extend $\nabla$ to $\g$ by linearity. Since $\{E_a\}_{a=1}^n$
  is a basis of $M$, every element $U\in M$ has a unique expression
  $U=E_aU^a$, and we extend $\nabla$ to $M$ through linearity and Leibniz'
  rule
  \begin{align*}
    \nablasub{d}U =\paraa{\nablasub{d}E_a}U^a + E_ad(U^a), 
  \end{align*}
  which then defines an affine connection on $(M,\g)$. From Proposition
  \ref{prop:Kozul.implies.metric.tf} it follows that
  $(M,h,\gphi,\nabla)$ is a pseudo-Riemannian calculus.
\end{proof}

\section{Curvature of pseudo-Riemannian calculi}
\label{sec:curvature.of.calculus}

\noindent 
In this section we will study symmetries of the
curvature tensor of a pseudo-Riemannian calculus, as well as
introduce an associated scalar curvature, in case it
exists. It turns out that in order to recover the full symmetry
(compared to the classical setting) of the curvature tensor,
one needs an extra assumption of hermiticity. Namely, although a real
connection calculus satisfies the requirement that
$h(\nablasub{d_1}E_1,E_2)$ is hermitian, there is no guarantee that
$h(\nablasub{d_1}\nablasub{d_2}E_1,E_2)$ is hermitian. However, with
this extra assumption, one may prove that all the familiar symmetries
of the curvature tensor hold
(cf. Proposition~\ref{prop:curvature.symmetries}). Pseudo-Riemannian
calculi fulfilling this extra condition will appear often in what
follows, and therefore we make the following definition.

\begin{definition}\label{def:real.pseudo.Riemannian.calculus}
  A pseudo-Riemannian calculus $(M,h,\gphi,\nabla)$ is said to be
  \emph{real} if $h(\nablasub{d_1}\nablasub{d_2}E_1,E_2)$ is hermitian
  for all $d_1,d_2\in\g$ and $E_1,E_2\in\Mphi$.
\end{definition}

\noindent
For later convenience, let us provide a slight reformulation of the
condition in the definition above.

\begin{lemma}\label{lemma:equiv.secnd.symmetry}
  Let $(M,h,\gphi,\nabla)$ be a pseudo-Riemannian calculus. Then the
  following statements are equivalent
  \begin{enumerate}
  \item $h(\nablasub{d_1}\nablasub{d_2}E_1,E_2)$ is hermitian
    for all $d_1,d_2\in\g$ and $E_1,E_2\in\Mphi$,
  \item $h(\nablasub{d_1}E_1,\nablasub{d_2}E_2)$ is hermitian
    for all $d_1,d_2\in\g$ and $E_1,E_2\in\Mphi$.
  \end{enumerate}
\end{lemma}

\begin{proof}
  Since the connection is metric, one may write
  \begin{align*}
    d_2h(\nablasub{d_1}E_1,E_2) =
    h(\nablasub{d_2}\nablasub{d_1}E_1,E_2)+h(\nablasub{d_1}E_1,\nablasub{d_2}E_2)
  \end{align*}
  Now, $d_2h(\nablasub{d_1}E_1,E_2)$ is hermitian (since $\nabla$ is
  real and $d_2$ is hermitian), and it follows that if one of
  $h(\nablasub{d_2}\nablasub{d_1}E_1,E_2)$ and
  $h(\nablasub{d_1}E_1,\nablasub{d_2}E_2)$ is hermitian, then the
  other one is hermitian too (since it is then a sum of two hermitian elements).
\end{proof}

\noindent
In a pseudo-Riemannian calculus $(M,h,\gphi,\nabla)$, one
introduces the curvature operator in a standard way as
\begin{align*}
  R(d_1,d_2)U = \nablasub{d_1}\nablasub{d_2}U-
  \nablasub{d_2}\nablasub{d_1}U-\nabla_{[d_1,d_2]}U
\end{align*}
for $d_1,d_2\in\g$ and $U\in M$.
The operator $R(d_1,d_2)$ has a trivial antisymmetry when exchanging
its arguments $d_1,d_2$ and, furthermore, due to the torsion-free condition, the
first Bianchi identity holds.

\begin{proposition}\label{prop:curvature.symmetries.first}
  Let $(M,h,\gphi,\nabla)$ be a pseudo-Riemannian calculus with
  curvature operator $R$. Then
  \begin{enumerate}
  \item\label{Rf.sym.trivial}
    $h(U,R(d_1,d_2)V) = -h(U,R(d_2,d_1)V)$
  \item\label{Rf.sym.bianchi}
    $R(d_1,d_2)\varphi(d_3)+R(d_2,d_3)\varphi(d_1)+R(d_3,d_1)\varphi(d_2)=0$,
  \end{enumerate}
  for $U,V\in M$ and $d_1,d_2,d_3\in\g$.
\end{proposition}

\begin{proof}
  Property (\ref{Rf.sym.trivial}) follows immediately from the
  definition of the curvature operator.
  To prove (\ref{Rf.sym.bianchi}), one uses the torsion free condition
  twice (set $E_a=\varphi(d_a)$):
  \begin{align*}
    R&(d_1,d_2)E_3 + R(d_2,d_3)E_1 + R(d_3,d_1)E_2 \\
    &= \nablasub{d_1}\paraa{\nablasub{d_2}E_3-\nablasub{d_3}E_2}
    +\nablasub{d_2}\paraa{\nablasub{d_3}E_1-\nablasub{d_1}E_3}
    +\nablasub{d_3}\paraa{\nablasub{d_1}E_2-\nablasub{d_2}E_1}\\
    &\quad - \nablasub{[d_1,d_2]}E_3
    -\nablasub{[d_2,d_3]}E_1-\nablasub{[d_3,d_1]}E_2\\
    &= \nablasub{d_1}\varphi([d_2,d_3])
    +\nablasub{d_2}\varphi([d_3,d_1])+\nablasub{d_3}\varphi([d_1,d_2])\\
    &\quad - \nablasub{[d_1,d_2]}E_3
    -\nablasub{[d_2,d_3]}E_1-\nablasub{[d_3,d_1]}E_2\\
    &=\varphi\paraa{[d_1,[d_2,d_3]]}+\varphi\paraa{[d_2,[d_3,d_1]]}
    +\varphi\paraa{[d_3,[d_1,d_2]]} = 0,
  \end{align*}
  where the last equality follows from the Jacobi identity, and the
  fact that $\varphi$ is a linear map.
\end{proof}

\noindent
As already mentioned, the full symmetry of the curvature operator is
recovered in the case of \emph{real} pseudo-Riemannian calculi. This
is stated in Proposition \ref{prop:curvature.symmetries}, and in the
proof we shall need the following short lemma.

\begin{lemma}\label{lemma:h.E.nabla.E}
  If $(M,h,\gphi,\nabla)$ is a pseudo-Riemannian calculus, then
  \begin{align*}
    d\paraa{h(E,E)} = 2h(E,\nablasub{d}E)
  \end{align*}
  for all $d\in\g$ and $E\in\Mphi$.
\end{lemma}

\begin{proof}
  Since $\nabla$ is a metric connection 
  \begin{align*}
    d\paraa{h(E,E)} = h\paraa{\nablasub{d}E,E}+h\paraa{E,\nablasub{d}E},
  \end{align*}
  and, as $\nabla$ is real, it follows that
  $h(E,\nablasub{d}E)=h(\nablasub{d}E,E)$, which implies that
  \begin{align*}
    d\paraa{h(E,E)} = 2h(E,\nablasub{d}E)
  \end{align*}
  for all $E\in\Mphi$ and $d\in\g$.
\end{proof}

\noindent 
Note that, for the sake of completeness, the results of
Proposition~\ref{prop:curvature.symmetries.first} are repeated in the
formulation below.

\begin{proposition}\label{prop:curvature.symmetries}
  Let $(M,h,\gphi,\nabla)$ be a real pseudo-Riemannian calculus, with
  curvature operator $R$. Then
  \begin{enumerate}
  \renewcommand{\theenumi}{\textup{(\alph{enumi})}}%
  \renewcommand{\labelenumi}{\theenumi}%
  \item \label{R.sym.trivial} 
    $h(U,R(d_1,d_2)V)=-h(U,R(d_2,d_1)V)$,
  \item \label{R.sym.interchange} 
    $h(E_1,R(d_1,d_2)E_2)=-h(E_2,R(d_1,d_2)E_1)$,
  \item \label{R.sym.bianchi} 
    $R(d_1,d_2)\varphi(d_3)+R(d_2,d_3)\varphi(d_1)+R(d_3,d_1)\varphi(d_2)=0$,
  \item \label{R.sym.interchange.pairs} 
    $h\paraa{\varphi(d_1),R(d_3,d_4)\varphi(d_2)}=
    h\paraa{\varphi(d_3),R(d_1,d_2)\varphi(d_4)}$,
  \end{enumerate}
  for all $U,V\in M$, $E_1,E_2\in\Mphi$ and $d_1,d_2,d_3,d_4\in\g$.
\end{proposition}

\begin{proof}
  Properties \ref{R.sym.trivial} and \ref{R.sym.bianchi} are contained
  in the statement of Proposition
  \ref{prop:curvature.symmetries.first}, which is valid for an
  arbitrary pseudo-Riemannian calculus.  Let now show that
  \ref{R.sym.interchange} holds, by proving that $h(E,R(d_1,d_2)E)=0$
  for all $E\in\Mphi$. By using the fact that $\nabla$ is metric, one
  computes
  \begin{align*}
    h\paraa{E,R(d_1,d_2)E} &=
    h\paraa{E,\nablasub{d_1}\nablasub{d_2}E-\nablasub{d_2}\nablasub{d_1}E-\nabla_{[d_1,d_2]}E} \\
    &=d_1h(E,\nablasub{d_2}E)-d_2h(E,\nablasub{d_1}E)-h(E,\nabla_{[d_1,d_2]}E),
  \end{align*}
  using the result in Lemma~\ref{lemma:equiv.secnd.symmetry} (and the
  fact that the pseudo-Riemannian calculus is assumed to be real). Next,
  it follows from Lemma~\ref{lemma:h.E.nabla.E} that
  \begin{align*}
    h\paraa{E,R(d_1,d_2)E} &=
    \half d_1d_2h(E,E)-\half d_2d_1h(E,E)-\half[d_1,d_2]h(E,E)=0.
  \end{align*}
  Finally, we prove \ref{R.sym.interchange.pairs} by using
  \ref{R.sym.bianchi} to write (again, $E_a=\varphi(d_a)$)
  \begin{align*}
    0 = &h\paraa{E_1,R(d_2,d_3)E_4+R(d_3,d_4)E_2+R(d_4,d_2)E_3}\\
    &+h\paraa{E_2,R(d_3,d_4)E_1+R(d_4,d_1)E_3+R(d_1,d_3)E_4}\\
    &+h\paraa{E_3,R(d_4,d_1)E_2+R(d_1,d_2)E_4+R(d_2,d_4)E_1}\\
    &+h\paraa{E_4,R(d_1,d_2)E_3+R(d_2,d_3)E_1+R(d_3,d_1)E_2}\\
    &= 2h(E_1,R(d_4,d_2)E_3) + 2h(E_2,R(d_1,d_3)E_4),
  \end{align*}
  by using \ref{R.sym.interchange} and \ref{R.sym.trivial}. Consequently, by
  using \ref{R.sym.trivial} once more, relation
  \ref{R.sym.interchange.pairs} follows.
\end{proof}

\subsection{Scalar curvature}

Let $(M,h,\gphi,\nabla)$ be a real pseudo-Riemannian calculus, and let
$\{\d_1,\ldots,\d_n\}$ be a basis of $\g$. Setting $E_a=\varphi(\d_a)$
one introduces the components of the metric and the curvature tensor
relative to this basis via
\begin{align*}
  h_{ab} &= h(E_a,E_b)\\
  R_{abpq} &= h\paraa{E_a,R(\d_p,\d_q)E_b},
\end{align*}
and we note that
$(R_{abpq})^\ast=R_{abpq}$ (using the fact that the pseudo-Riemannian
calculus is real). Proposition~\ref{prop:curvature.symmetries}
implies that
\begin{align}
  &R_{abpq}=-R_{abqp},\label{eq:R.sym.comp.trivial}\\
  &R_{abpq} = -R_{bapq},\label{eq:R.sym.comp.interchange}\\
  &R_{abpq}=R_{pqab},\label{eq:R.sym.comp.interchange.pairs}\\
  &R_{apqr}+R_{aqrp}+R_{arpq}=0.\label{eq:R.sym.comp.bianchi}
\end{align}
In the traditional definition of scalar curvature
$S=h^{ab}h^{pq}R_{apbq}$ one makes use of the inverse of the metric to
contract indices of the curvature tensor.  For an arbitrary algebra,
the metric $h_{ab}$ may fail to be invertible; i.e., there does not
exist $h^{ab}$ such that $h^{ab}h_{bc}=\delta^a_c\mid$. However, one
might be in the situation where there exist $H\in\A$ and $\hhat^{ab}$
such that
\begin{align*}
  \hhat^{ab}h_{bc} =h_{cb}\hhat^{ba} = \delta^a_cH.
\end{align*}
If $H$ is hermitian and regular (i.e. not a zero divisor), then we
say that $h_{ab}$ has a pseudo-inverse $(\hhat^{ab},H)$. 

\begin{lemma}\label{lemma:hhat.properties}
  If $(\hhat^{ab},H)$ and $(\hat{g}^{ab},G)$ are pseudo-inverses for $h_{ab}$
  then
  \begin{enumerate}
  \item if $G=H$ then $\hat{g}^{ab}=\hhat^{ab}$,\label{pseudo.inv.GeqH}
  \item $[h_{ab},H]=[\hhat^{ab},H]=0$,\label{pseudo.inv.h.com.H}
  \item $\paraa{\hhat^{ab}}^\ast = \hhat^{ba}$,\label{pseudo.inv.hhat.sym}
  \item $\hat{g}^{ab}H=G\hhat^{ab}$ and $H\hat{g}^{ab}=\hhat^{ab}G$,\label{pseudo.inv.Hg.com}
  \item if $[H,\hat{g}^{ab}]=0$ then $[G,\hhat^{ab}]=[H,G]=0$.\label{pseudo.inv.HG.com}
  \end{enumerate}
\end{lemma}

\begin{proof}
  To prove (\ref{pseudo.inv.GeqH}), one assumes that $(\hhat,H)$ and $(\hat{g},H)$ are
  pseudo-inverses of $h$. Then it follows that
  $\paraa{\hat{g}^{ab}-\hhat^{ab}}h_{bc} = 0$
  which, when multiplying from the right by $\hhat^{cp}$, yields
  \begin{align*}
    \paraa{\hat{g}^{ap}-\hhat^{ap}}H = 0.
  \end{align*}
  Since $H$ is a regular element it follows that
  $\hat{g}^{ap}=\hhat^{ap}$.

  By using the definition of the two pseudo-inverses, one may rewrite
  the expression $\hat{g}^{ab}h_{bc}\hhat^{cp}$ in two ways,
  \begin{align*}
    &\paraa{\hat{g}^{ab}h_{bc}}\hhat^{cp} = G\delta^a_c\hhat^{cp} =
    G\hhat^{ap}\\
    &\hat{g}^{ab}\paraa{h_{bc}\hhat^{cp}} = \hat{g}^{ab}H\delta^p_b
    =\hat{g}^{ap}H
  \end{align*}
  proving (\ref{pseudo.inv.Hg.com}) (consider
  $\hhat^{ab}h_{bc}\hat{g}^{cp}$ for the second part of the
  statement). Setting $(\hhat,H)=(\hat{g},G)$ in the above result
  immediately gives $[H,\hhat^{ab}]=0$. Together with
  \begin{align*}
    h_{ab}H = h_{ap}\delta^{p}_bH = h_{ap}\hhat^{pc}h_{cb}
    =H\delta_a^ch_{cb}=Hh_{ab}
  \end{align*}
  this proves (\ref{pseudo.inv.h.com.H}). 

  Let us consider property (\ref{pseudo.inv.HG.com}). If
  $[H,\hat{g}^{ab}]=0$ then (\ref{pseudo.inv.Hg.com}) implies that
  \begin{align*}
    G\hhat^{ab} = \hat{g}^{ab}H = H\hat{g}^{ab} = \hhat^{ab}G.
  \end{align*}
  Moreover, 
  \begin{align*}
    &\hat{g}^{ab}H-H\hat{g}^{ab}=0\implies
    h_{ca}\hat{g}^{ab}H-h_{ca}H\hat{g}^{ab}=0\implies
    \text{(using (\ref{pseudo.inv.h.com.H}))}\\
    &h_{ca}\hat{g}^{ab}H-Hh_{ca}\hat{g}^{ab}=0\implies
    \paraa{GH-HG}\delta^{b}_c = 0\implies [G,H]=0,
  \end{align*}
  which concludes the proof of (\ref{pseudo.inv.HG.com}).

  Finally, to prove (\ref{pseudo.inv.hhat.sym}) one considers the
  hermitian conjugates of $h_{ab}\hhat^{bc}=H\delta_a^c$ and
  $\hhat^{ab}h_{bc}=H\delta^a_c$, which gives
  \begin{align*}
    &\paraa{\hhat^{bc}}^\ast h_{ba}=H\delta^{c}_a\\
    &h_{cb}\paraa{\hhat^{ab}}^\ast = H\delta^a_c,
  \end{align*}
  by using that $h_{ab}^\ast=h_{ba}$. The above equations show that
  if $k^{ab}=(\hhat^{ba})^\ast$ then $(k^{ab},H)$ is a pseudo-inverse
  for $h_{ab}$. Since $(\hhat^{ab},H)$ and $(k^{ab},H)$ are
  pseudo-inverses for $h_{ab}$, it follows from
  (\ref{pseudo.inv.GeqH}) that $\hhat^{ab}=k^{ab}=(\hhat^{ba})^\ast$. 
\end{proof}

\begin{definition}
  Let $(M,h,\gphi,\nabla)$ be a real pseudo-Riemannian calculus such that
  $h_{ab}$ has a pseudo-inverse $(\hhat^{ab},H)$ with respect to a
  basis of $\g$. A \emph{scalar curvature of
    $(M,h,\gphi,\nabla)$ with respect to $(\hhat^{ab},H)$} is an element $S\in\A$ such that
  \begin{align*}
    \hhat^{ab}R_{apbq}\hhat^{pq} = HSH.
  \end{align*}
\end{definition}

\begin{remark}
  Note that it is easy to show that $\hhat^{ab}R_{apbq}\hhat^{pq}$
  and, hence, the scalar curvature with respect to $(\hhat^{ab},H)$,
  is independent of the choice of basis in $\g$.
\end{remark}

\begin{proposition}
  Let $(M,h,\gphi,\nabla)$ be a real pseudo-Riemannian calculus, and
  let $(\hhat^{ab},H)$ be a pseudo-inverse of $h_{ab}$ with respect to
  a basis of $\g$. Then there exists at most
  one scalar curvature of $(M,h,\gphi,\nabla)$ with respect to
  $(\hhat^{ab},H)$ and, furthermore, the scalar curvature is
  hermitian.
\end{proposition}

\begin{proof}
  Uniqueness of the scalar curvature follows immediately from the fact
  that $H$ is regular; namely,
  \begin{align*}
    HSH = HS'H\equivalent H(S-S')H=0,
  \end{align*}
  which then implies that $S=S'$ by the regularity of $H$. 

  As noted in the beginning of this section, $R_{abcd}$
  is hermitian. Furthermore, Lemma~\ref{lemma:hhat.properties} states
  that $\paraa{\hhat^{ab}}^\ast=\hhat^{ba}$ which implies that
  \begin{align*}
    \paraa{\hhat^{ab}R_{apbq}\hhat^{pq}}^\ast = 
    \hhat^{qp}R_{apbq}\hhat^{ba} = 
    \hhat^{qp}R_{qbpa}\hhat^{ba}=\hhat^{ab}R_{apbq}\hhat^{pq},
  \end{align*}
  by using
  \eqref{eq:R.sym.comp.trivial}--\eqref{eq:R.sym.comp.interchange.pairs}. From
  the definition of scalar curvature, this implies that
  \begin{align*}
    HSH = (HSH)^\ast = HS^\ast H\equivalent H(S-S^\ast)H = 0.
  \end{align*}
  Since $H$ is assumed to be regular, it follows that $S=S^\ast$.
\end{proof}

\noindent If $S$ is the scalar curvature with respect to a
pseudo-inverse $(\hhat^{ab},H)$, in which $H$ is central, then any
scalar curvature (with respect to an arbitrary pseudo-inverse)
coincides with $S$, giving a unique hermitian scalar curvature of a real
pseudo-Riemannian calculus.

\begin{proposition}\label{prop:H.central.unique.scalar.curvature}
  Let $(M,h,\gphi,\nabla)$ be a real pseudo-Riemannian calculus with
  scalar curvature $S$ with respect to $(\hhat^{ab},H)$. If $H\in\ZA$
  then the scalar curvature is unique; i.e. if $S'$ is the scalar
  curvature with respect to $(\hat{g}^{ab},G)$, then $S'=S$.
\end{proposition}

\begin{proof}
  If $H\in\ZA$ then property (\ref{pseudo.inv.Hg.com}) of
  Lemma~\ref{lemma:hhat.properties} implies that
  \begin{align*}
    G(HSH)G = G\hhat^{ab}R_{apbq}\hhat^{pq}G
    =H\hat{g}^{ab}R_{apbq}\hat{g}^{pq}H=H(GS'G)H,
  \end{align*}
  and since $[H,G]=0$ one obtains
  \begin{align*}
    HG(S-S')GH = 0\implies S=S'
  \end{align*}
  since $G$ and $H$ are assumed to be regular.
\end{proof}

\begin{remark}
  In particular, if the metric $h_{ab}$ is invertible, i.e. it has a
  pseudo-inverse $(\hhat^{ab},\mid)$, then
  Proposition~\ref{prop:H.central.unique.scalar.curvature} implies
  that there exists a unique scalar curvature of the corresponding
  real pseudo-Riemannian calculus.
\end{remark}

\section{The noncommutative torus}
\label{sec:nc.torus}

\noindent After having developed a general framework for Riemannian
curvature of a real metric calculus, it is time to consider some
examples, in order to motivate our definitions. As a starter, let us
consider the noncommutative torus and construct a pseudo-Riemannian
calculus over it (cf. \cite{a:curvatureGeometric}
for a related approach that uses the concrete embedding of the torus
into $\reals^4$). For the noncommutative torus, our construction of a
Levi-Civita connection, and its corresponding curvature, is similar to
the approach taken in \cite{r:leviCivita}.

As we shall work in close analogy with differential geometry, let us
briefly review the geometry of the Clifford torus. We consider the
Clifford torus as embedded in $\reals^4$ with the flat induced
metric. Concretely, let us consider the following parametrization
\begin{align*}
  \xv = (x^1,x^2,x^3,x^4) = (\cos u,\sin u,\cos v,\sin v)
\end{align*}
which implies that the tangent space at a point is spanned by
\begin{align*}
  &\d_u\xv = (-\sin u,\cos u,0,0)=(-x^2,x^1,0,0)\\
  &\d_v\xv = (0,0,-\sin v,\cos v)=(0,0,-x^4,x^3),
\end{align*}
from which the induced metric is obtained as
\begin{align*}
  (h_{ab}) = 
  \begin{pmatrix}
    1 & 0 \\ 0 & 1
  \end{pmatrix}.
\end{align*}
Setting $z=x^1+ix^2$, $w=x^3+ix^4$ and
$\d_1=\d_u$, $\d_2=\d_v$ yields
\begin{alignat*}{2}
  &\d_1z = iz &\qquad &\d_1w = 0\\
  &\d_2z = 0  &       &\d_2w = iw.
\end{alignat*}
As the noncommutative torus $\Ttheta$, we consider the unital
$\ast$-algebra generated by two unitary operators $Z,W$ satisfying $WZ
= qZW$ with $q = e^{2\pi i\theta}$, and we introduce
\begin{alignat*}{2}
  X^1 &= \frac{1}{2}\paraa{Z+\Zs}    &\qquad X^2
  &=\frac{1}{2i}\paraa{Z-\Zs}\\
  X^3 &= \frac{1}{2}\paraa{W+\Ws}    &\qquad X^4 
  &=\frac{1}{2i}\paraa{W-\Ws}.
\end{alignat*}
In analogy with the geometrical setting, let $M$ be the (right)
submodule of $(\Ttheta)^4$ generated by
\begin{align*}
  &E_1 = \paraa{-X^2,X^1,0,0}\\
  &E_2 = \paraa{0,0,-X^4,X^3},
\end{align*}
and for $U,V\in M$, with $U=E_aU^a$ and $V=E_aV^a$ we set
\begin{align*}
  h(U,V) = \sum_{a=1}^2\paraa{U^a}^\ast V^a.
\end{align*}

\begin{proposition}
  The elements $E_1,E_2\in M$ gives a basis for $M$ and $h$ is a
  nondegenerate hermitian form on $M$. Thus, $(M,h)$ is a free metric $\Ttheta$-module.
\end{proposition}

\begin{proof}
  First, let us show that $E_1,E_2$ are free generators
  \begin{align*}
    &E_1a + E_2b = 0\implies
    (-X^2a,X^1a,-X^4b,X^3b) = (0,0,0,0)\implies\\
    &\begin{cases}
      \paraa{(X^1)^2 + (X^2)^2}a = 0\\
      \paraa{(X^3)^2 + (X^4)^2}b = 0
    \end{cases}\equivalent
    \begin{cases}
      Z\Zs a = 0\\
      W\Ws b = 0
    \end{cases}\equivalent
    a=b=0.
  \end{align*}
  Next, we prove that $h$ is nondegenerate on $M$. Let $U,V\in M$ and
  write $U=E_aU^a$. Assuming that $h(U,V)=0$ for all $V\in M$ may be
  equivalently stated as $h(U,E_a)=0$ for $a=1,2$, which immediately
  gives $U^1=U^2 = 0$. 
\end{proof}

\noindent
Next, we let $\g$ be the (real) Lie algebra generated by the two
hermitian derivations $\d_1,\d_2$, given by
\begin{alignat*}{2}
  &\d_1Z = iZ  &\qquad &\d_1W = 0\\
  &\d_2 Z = 0 &        &\d_2W = iW,
\end{alignat*}
from which it follows that $[\d_1,\d_2] = 0$. Together with the map
$\varphi:\g\to M$, defined as $\varphi(\d_a) = E_a$ and extended by
linearity, it is easy to check that $(M,h,\gphi)$ is a real metric
calculus over $\Ttheta$. Furthermore, we note that, with respect to the basis
$\{\d_1,\d_2\}$ of $\g$, the metric
\begin{align*}
  (h_{ab}) = \paraa{h(E_a,E_b)} = 
  \begin{pmatrix}
    \mid & 0 \\
    0 & \mid
  \end{pmatrix}
\end{align*}
is invertible.

One is now in position to use Corollary
\ref{cor:connection.from.Kozul} to find a unique connection $\nabla$
on $M$ such that $(M,h,\gphi,\nabla)$ is a pseudo-Riemannian calculus.
However, since $h(E_a,E_b)=\delta_{ab}\mid$ and $[\d_a,\d_b]=0$, the
only solution to
\begin{equation*}
  \begin{split}
    2h&(U_{ab},E_c)=
    \d_ah(E_b,E_c)+\d_bh(E_a,E_c)-\d_ch(E_a,E_b)\\
    &\quad
    -h\paraa{E_a,\varphi([\d_b,\d_c])}
    +h\paraa{E_b,\varphi([\d_c,\d_a])}
    +h\paraa{E_c,\varphi([\d_a,\d_b])}
  \end{split}
\end{equation*}
is $U_{ab}=0$, which gives $\nabla_dU = 0$ for all $d\in\g$ and $U\in
M$. Hence, the curvature of the corresponding pseudo-Riemannian
calculus vanishes identically, and the (unique) scalar curvature is $0$.

As done in \cite{r:leviCivita} one can obtain more interesting results
by conformally perturbing the metric 
\begin{align*}
  h_\alpha(U,V) = \sum_{a=1}^2\paraa{U^a}^\ast e^{\alpha} V^a
\end{align*}
for some hermitian element $\alpha\in \Ttheta$ (here, of course, one
considers the smooth part of the $C^\ast$-algebra generated by $Z,W$).
One can easily check that $(M,h_\alpha,\gphi)$ is a real metric
calculus, and one may find a connection $\nabla$ (using
Corollary~\ref{cor:connection.from.Kozul}) such that
$(M,h_\alpha,\gphi,\nabla)$ is a pseudo-Riemannian calculus.  However,
unless $\alpha$ is central, it will in general not be a \emph{real}
pseudo-Riemannian calculus.

\section{The noncommutative 3-sphere}
\label{sec:nc.three.sphere}

\noindent As a main motivating example for this paper, we consider the
noncommutative 3-sphere. We shall explicitly construct a real
pseudo-Riemannian calculus together with a basis of $\g$ for which the
metric $h_{ab}$ has a pseudo-inverse $(\hhat^{ab},H)$ with $H\in\ZA$,
giving a unique scalar curvature (via
Proposition~\ref{prop:H.central.unique.scalar.curvature}). As for the
case of the torus, we shall use the analogy with differential geometry
to find an appropriate metric calculus. Therefore, let us start by
recalling the Hopf parametrization of $S^3$.

\subsection{Hopf coordinates for $S^3$}
\label{sec:hopf.coordinates}

The 3-sphere can be described as embedded in $\complex^2$ by two
complex coordinates $z=x^1+ix^2$ and $w=x^3+ix^4$, satisfying
$|z|^2+|w|^2=1$, which can be realized by
\begin{align*}
  &z = e^{i\xi_1}\sin\eta\\
  &w = e^{i\xi_2}\cos\eta,
\end{align*}
giving
\begin{alignat*}{2}
  &x^1=\cos\xi_1\sin\eta  & \qquad  &x^2=\sin\xi_1\sin\eta\\
  &x^3= \cos\xi_2\cos\eta &        &x^4=\sin\xi_2\cos\eta.
\end{alignat*}
At every point where $0<\xi_1,\xi_2<2\pi$ and $0<\eta<\pi/2$, the
tangent space is spanned by the three vectors
\begin{align*}
  &E_1 = \d_1(x^1,x^2,x^3,x^4) = (-x^2,x^1,0,0)\\
  &E_2 = \d_2(x^1,x^2,x^3,x^4) = (0,0,-x^4,x^3)\\
  &E_\eta = \d_\eta(x^1,x^2,x^3,x^4) =
    (\cos\xi_1\cos\eta,\sin\xi_1\cos\eta,
    -\cos\xi_2\sin\eta,-\sin\xi_2\sin\eta),
\end{align*}
where $\d_1=\d_{\xi_1}$ and $\d_2=\d_{\xi_2}$. Instead of $\d_\eta$,
one may introduce the derivation $\d_3 = |z||w|\d_\eta$, which gives
\begin{align*}
  E_3 = \d_3(x^1,x^2,x^3,x^4) = 
  (x^1|w|^2,x^2|w|^2,-x^3|z|^2,-x^4|z|^2),
\end{align*}
and one may equally well span the tangent space by $E_1,E_2,E_3$. The
action of $\d_1,\d_2,\d_3$ on $z$ and $w$ is given by
\begin{alignat}{2}
  &\d_1(z)=iz &\qquad &\d_1(w)=0\label{eq:Sthree.d.one}\\
  &\d_2(z)=0 &        &\d_2(w)=iw\label{eq:Sthree.d.two}\\
  &\d_3(z)=z|w|^2 &   &\d_3(w)=-w|z|^2.\label{eq:Sthree.d.three}
\end{alignat}

\noindent With respect to the basis $\{E_1,E_2,E_3\}$ of $T_pS^3$ the induced metric becomes
\begin{align}
  (h_{ab}) = 
  \begin{pmatrix}
    |z|^2 & 0 & 0\\
    0 & |w|^2 & 0\\
    0 & 0 & |z|^2|w|^2
  \end{pmatrix}.
\end{align}

\subsection{A pseudo-Riemannian calculus for $\Sthreet$}
\label{sec:pseudo.Riemannian.calculus.Sthreet}

\noindent
For our purposes, the noncommutative three sphere $\Sthreet$
\cite{m:ncspheres,m:ncspheresII} is a unital $\ast$-algebra generated
by $Z,\Zs,W,\Ws$ subject to the relations
\begin{alignat}{4}
  &WZ = qZW  &\quad  &\Ws Z=\qb Z\Ws &\quad &W\Zs=\qb\Zs W & &\Ws\Zs=q\Zs\Ws\label{eq:Sthree.def}\\
  &\Zs Z=Z\Zs & &\Ws W=W\Ws & &W\Ws = \mid- Z\Zs, \notag
\end{alignat}
where $q=e^{2\pi i\theta}$. It follows from the defining relations
that a basis for $\Sthreet$ is given by the monomials
\begin{align*}
  Z^i(Z^\ast)^jW^{(k)}
\end{align*}
for $i,j\geq 0$ and $k\in\integers$, where 
\begin{align*}
  W^{(k)} = 
  \begin{cases}
    W^k\text{ if }k\geq 0\\
    (W^\ast)^{-k}\text{ if }k<0
  \end{cases}.
\end{align*}

\noindent
Let us collect a few properties of $\Sthreet$ that will be useful to
us.

\begin{proposition}\label{prop:Sthree.zero.div}
  If $a\in\Sthreet$ then
  \begin{enumerate}
  \item $ZZ^\ast a = 0\implies a=0$,
  \item $WW^\ast a = 0\implies a=0$.
  \end{enumerate}
  Moreover, $ZZ^\ast$ and $WW^\ast$ are central elements of $\Sthreet$.
\end{proposition}

\begin{proof}
  Let us prove that $Z\Zs$ commutes with every element of $\Sthreet$ (the
  proof for $W\Ws$ is analogous). From the defining relations of the
  algebra, it is clear that $Z\Zs$ commutes with $Z$ and $\Zs$. Let us
  check that $Z\Zs$ commutes with $W$ and $\Ws$:
  \begin{align*}
    &WZ\Zs = qZW\Zs = q\qb Z\Zs W = Z\Zs W\\
    &\Ws Z\Zs = \qb Z\Ws\Zs =\qb q Z\Zs\Ws = Z\Zs\Ws.
  \end{align*}
  Next, let us show that neither $ZZ^\ast$ nor $WW^\ast$ is a zero
  divisor.  An arbitrary element $a\in\Sthreet$ may be written as
  \begin{align*}
    a = \sum_{i,j\geq 0,k\in\integers}a_{ijk}Z^i(Z^\ast)^jW^{(k)}
  \end{align*}
  for $a_{ijk}\in\complex$, and it follows that
  \begin{align*}
    ZZ^\ast a = \sum_{i,j\geq 0,k\in\integers}a_{ijk}Z^{i+1}(Z^\ast)^{j+1}W^{(k)}
  \end{align*}
  since $[Z,Z^\ast] = 0$. As $Z^i(Z^\ast)^jW^{(k)}$ is a basis for
  $\Sthreet$, setting $ZZ^\ast a=0$ demands that $a_{ijk}=0$ for all
  $i,j\geq 0$ and $k\in\integers$, which implies that $a=0$. Analogously, 
  \begin{align*}
    WW^\ast a &= (\mid-ZZ^\ast)a 
    =\sum_{i,j\geq 0,k\in\integers}
     \parab{a_{ijk}Z^{i}(Z^\ast)^{j}W^{(k)}-a_{ijk}Z^{i+1}(Z^\ast)^{j+1}W^{(k)}}\\
    &=\sum_{j\geq 0,k\in\integers}a_{0jk}(Z^\ast)^jW^{(k)}
      +\sum_{i\geq 1,k\in\integers}a_{i0k}Z^iW^{(k)}\\
    &\quad+\sum_{i,j\geq 1,k\in\integers}
    \paraa{a_{ijk}-a_{i-1,j-1,k}}Z^{i}(Z^\ast)^{j}W^{(k)},
  \end{align*}
  which can easily be seen to give $a_{ijk}=0$ upon setting $WW^\ast a=0$.
\end{proof}

\noindent
Let us introduce the notation 
\begin{alignat*}{2}
  &X^1 = \half\paraa{Z + Z^\ast} &\qquad
  &X^2 = \frac{1}{2i}\paraa{Z-Z^\ast}\\
  &X^3 = \half\paraa{W + W^\ast} &\qquad
  &X^4 = \frac{1}{2i}\paraa{W-W^\ast}\\
  &\Zsq = ZZ^\ast & &\Wsq = WW^\ast,
\end{alignat*}
and note that $\Zsq = (X^1)^2 + (X^2)^2$ and 
$\Wsq=(X^3)^2 + (X^4)^2$, as well as
\begin{align*}
  (X^1)^2 + (X^2)^2 + (X^3)^2 + (X^4)^2 = \Zsq+\Wsq = \mid.
\end{align*}
In the following, we shall construct a real pseudo-Riemannian calculus
for $\Sthreet$. Let us start by introducing a metric module $(M,h)$
in close analogy with the Hopf parametrization in Section
\ref{sec:hopf.coordinates}.  Therefore, we let $E_1,E_2,E_3$ be the
following elements of the free (right) module $(\Sthreet)^4$:
\begin{equation}\label{Sthree.Evec.def}
  \begin{split}
    &E_1 = (-X^2,X^1,0,0)\\
    &E_2 = (0,0,-X^4,X^3)\\
    &E_3 = (X^1\Wsq, X^2\Wsq, -X^3\Zsq, -X^4\Zsq),
  \end{split}
\end{equation}
and let $M$ be the module generated by $\{E_1,E_2,E_3\}$.

\begin{proposition}\label{prop:M.free.module}
  The module $M=\{E_1a+E_2b+E_3c:\,a,b,c\in\Sthreet\}$ is a free
  right $\Sthreet$-module with a basis given by the elements
  $\{E_1,E_2,E_3\}$.
\end{proposition}

\begin{proof}
  By construction $\{E_1,E_2,E_3\}$ are generators of $M$. To prove
  that $M$ is a free module, we assume that $a,b,c\in\Sthreet$ is such
  that
  \begin{align*}
    E = E_1a + E_2b + E_3c = 0,
  \end{align*}
  and show that $a=b=c=0$. The requirement that $E=0$ is equivalent to
  \begin{alignat*}{2}
    -&X^2a + X^1\Wsq c = 0 &\qquad 
    &X^1a + X^2\Wsq c = 0 \\
    -&X^4b - X^3\Zsq c = 0 &
    &X^3b - X^4\Zsq c  = 0,
  \end{alignat*}
  and multiplying the first to equations by $X^1$ and $X^2$,
  respectively, and summing them yields (using that $[X^1,X^2]=0$)
  \begin{align*}
    \paraa{(X^1)^2 + (X^2)^2}\Wsq c = 0\equivalent
    \Zsq\Wsq c = 0.
  \end{align*}
  It follows from Proposition \ref{prop:Sthree.zero.div} that $c=0$,
  and the system of equations becomes
  \begin{alignat*}{2}
    &X^2a = 0 & \qquad
    &X^1a = 0\\
    &X^4b = 0 & 
    &X^3b = 0,
  \end{alignat*}
  from which it follows that $\paraa{(X^1)^2+(X^2)^2}a = 0$ and
  $\paraa{(X^3)^2+(X^4)^2}a = 0$, which is equivalent to
  \begin{align*}
    \Zsq a = 0\qquad \Wsq b = 0.
  \end{align*}
  Again, it follows from Proposition \ref{prop:Sthree.zero.div} that
  $a=b=0$. This shows that $\{E_1,E_2,E_3\}$ is basis for $M$.
\end{proof}

\noindent 
In the differential geometric setting, the three tangent vectors
$E_1,E_2,E_3$ are associated to the three derivations $\d_1,\d_2,\d_3$,
as given in \eqref{eq:Sthree.d.one}--\eqref{eq:Sthree.d.three}. It
turns out that these derivations have noncommutative analogues.

\begin{proposition}\label{prop:Sthree.derivations}
  There exist hermitian derivations $\d_1,\d_2,\d_3\in\Der(\Sthreet)$
  such that
  \begin{alignat*}{2}
    &\d_1(Z) = iZ & \qquad &\d_1(W) = 0\\
    &\d_2(Z) = 0  & &\d_2(W) = iW\\
    &\d_3(Z) = Z\Wsq & &\d_3(W) = -W\Zsq,
  \end{alignat*}
  and $[\d_a,\d_b]=0$ for $a,b=1,2,3$.
\end{proposition}

\begin{proof}
  Let us show that $\d_3$ exists; the proof that $\d_1,\d_2$ exist is
  analogous. If $\d_3$ exists, the fact that it is hermitian, together
  with $\d_3(Z)=Z\Wsq$ and $\d_3(W)=-W\Zsq$ completely determines
  $\d_3$ via 
  \begin{align*}
    &\d_3(Z) = Z\Wsq\qquad \d_3(W)=-W\Zsq\\
    &\d_3(\Zs) = \Zs\Wsq\qquad \d_3(\Ws) = -\Ws\Zsq,
  \end{align*}
  since the action on an arbitrary element of $\Sthreet$ is given by
  applying Leibniz' rule repeatedly. Conversely, one may try to define
  $\d_3$ via the above relations and extend it to $\Sthreet$ through
  Leibniz' rule. However, to show that $\d_3$ is a derivation on
  $\Sthreet$, one one needs to check that it respects all the
  relations between $Z,W$ (given in \eqref{eq:Sthree.def}). For
  instance, applying Leibniz' rule to $\d_3(WZ-qZW)$ gives
  \begin{align*}
    \d_3(WZ-qZW) &= (\d_3W)Z + W(\d_3Z) - q(\d_3Z)W -qZ(\d_3W)\\
    &= -W\Zsq Z + WZ\Wsq-qZ\Wsq W+qZW\Zsq\\
    &=-(WZ-qZW)\Zsq+(WZ-qZW)\Wsq = 0,
  \end{align*}
  as required (using that $\Zsq$ and $\Wsq$ are central). In the same
  way, one may check that $\d_3$ is compatible with all the relations
  in $\Sthreet$ (given in \eqref{eq:Sthree.def}), which shows that $\d_3$ is indeed a derivation on
  $\Sthreet$.  To prove that $[\d_a,\d_b]=0$ one simply shows that
  \begin{align*}
    [\d_a,\d_b](Z)=[\d_a,\d_b](\Zs)=
    [\d_a,\d_b](W)=[\d_a,\d_b](\Ws)=0,
  \end{align*}
  which, by Leibniz' rule, implies that $[\d_a,\d_b](a)=0$ for all
  $a\in\Sthreet$. For instance
  \begin{align*}
    [\d_1,\d_3](Z) &= \d_1\paraa{\d_3(Z)}-\d_3\paraa{\d_1(Z)}
    =\d_1\paraa{Z\Wsq}-\d_3\paraa{iZ} \\
    &=\d_1(Z)\Wsq + Z\d_1(\Wsq)-iZ\Wsq=Z\d_1(W\Ws)=0.
  \end{align*}
  The remaining computations are carried out in the same manner,
  all giving $0$.
\end{proof}

\noindent
Next, let us construct a real metric calculus over $\Sthreet$. As
the metric module we choose the free module $M$ defined
in Proposition \ref{prop:M.free.module}, together with the
hermitian form
\begin{align*}
  h(U,V) = \sum_{a,b=1}^3(U^a)^\ast h_{ab}V^b
\end{align*}
where $U=E_aU^a$, $V=E_aV^a$ and 
\begin{align}\label{eq:Sthree.metric}
  (h_{ab}) =
  \begin{pmatrix}
    \Zsq & 0 & 0 \\
    0 & \Wsq & 0 \\
    0 & 0 & \Zsq\Wsq
  \end{pmatrix}.
\end{align}
(Note that $h$ is induced from the canonical metric on the
free module $(\Sthreet)^4$; i.e. $h_{ab}=\sum_{i=1}^4(E_a^i)^\ast E_b^i$,
where $E_a = e_iE_a^i$.) Furthermore, we let $\g$ be the (abelian) Lie algebra
generated by the derivations $\d_1,\d_2,\d_3$ (in Proposition
\ref{prop:Sthree.derivations}) and set $\varphi(\d_a)=E_a$ (and extend
it as a linear map over $\reals$).

\begin{proposition}
  $(M,h,\gphi)$ is a real metric calculus over $\Sthreet$.
\end{proposition}

\begin{proof}
  Let us first prove that $(M,h)$ is a metric module. From the
  definition of $h$, it is clear that $h$ is a hermitian form, and it
  remains to show that it is non-degenerate. Assume that $h(U,V)=0$
  for all $V\in M$. In particular, one may choose $V=E_a$, which gives
  \begin{align*}
    &0=h(U,E_1) = (U^1)^\ast h_{11} = (U^1)^\ast\Zsq\\
    &0=h(U,E_2) = (U^2)^\ast h_{22} = (U^2)^\ast\Wsq\\
    &0=h(U,E_3) = (U^3)^\ast h_{33} = (U^3)^\ast\Zsq\Wsq,
  \end{align*}
  and from Proposition \ref{prop:Sthree.zero.div} it follows that
  $U^1=U^2=U^3 = 0$. Hence, $h$ is non-degenerate, which shows that
  $(M,h)$ is a metric module.

  Moreover, it is clear that $\varphi(\g)$ generates $M$ since
  $E_a=\varphi(\d_a)$, for $a=1,2,3$, is in the image of
  $\varphi$. Finally, for $E,E'\in\Mphi$, it is easy to see that
  $h(E,E')$ is hermitian since $h_{ab}$ is central and hermitian
  (cf. (\ref{eq:Sthree.metric})).
\end{proof}

\noindent
Since $M$ is a free module, and $E_a=\varphi(\d_a)$ is a basis
for $M$, one may use Corollary~\ref{cor:connection.from.Kozul} to
construct a metric and torsion-free connection on $(M,h,\gphi)$.

\begin{proposition}\label{prop:S3.pseudo.Riemannian.calculus}
  There exists a (unique) connection $\nabla$ on $(M,h,\gphi)$ such that
  $(M,h,\gphi,\nabla)$ is a real pseudo-Riemannian calculus. The connection
  is given by
  \begin{alignat*}{3}
    &\nabla_1E_1 = -E_3 &\qquad& \nabla_1E_2 = 0&\qquad&\nabla_1E_3 = E_1\Wsq\\
    &\nabla_2E_1 = 0&\qquad&\nabla_2E_2=E_3&\qquad&\nabla_2E_3=-E_2\Zsq\\
    &\nabla_3E_1 = E_1\Wsq &\qquad& \nabla_3E_2=-E_2\Zsq&\qquad&\nabla_3E_3=E_3(\Wsq-\Zsq),
  \end{alignat*}
  where $\nabla_a\equiv\nablasub{\d_a}$. 

\end{proposition}

\begin{proof}
  It is clear that $(M,h,\gphi)$ satisfies the prerequisites of
  Corollary~\ref{cor:connection.from.Kozul}. Furthermore, it is a
  straightforward exercise to check that $U_{ab}=\nabla_aE_b$
  satisfy equation \eqref{eq:Kozul.Uab}, which then implies that there
  exists a connection $\nabla$ on $(M,\g)$, given by $\nabla_aE_b$
  above, such that $(M,h,\gphi)$ is a pseudo-Riemannian calculus.

  Let us now show that the pseudo-Riemannian calculus is real; i.e,
  that the elements $h(\nabla_a\nabla_bE_p,E_q)$ are hermitian for all
  $a,b,p,q\in\{1,2,3\}$. We introduce the connection coefficients
  $\Gamma_{ab}^c\in\Sthreet$ through
  \begin{align*}
    \nabla_aE_b = E_c\Gamma_{ab}^c,
  \end{align*}
  and note that $\Gamma_{ab}^c$ is central and hermitian for all
  $a,b,c\in\{1,2,3\}$. It follows that
  \begin{align*}
    &\nabla_a\nabla_bE_p = \nabla_a\paraa{E_r\Gamma_{bp}^r}
    =\paraa{\nabla_aE_r}\Gamma_{bp}^r + E_r\d_a\Gamma_{bp}^r\implies\\
    &h(\nabla_a\nabla_bE_p,E_q) = h(\nabla_aE_r,E_q)\Gamma^{r}_{bp}
    +h(E_r,E_q)\paraa{\d_a\Gamma^r_{bp}}.
  \end{align*}
  Since $(M,h,\gphi,\nabla)$ is a pseudo-Riemannian calculus and
  $\Gamma^r_{bp}$ is central and hermitian, it follows that the first
  term is hermitian. Furthermore, since $\d_a$ is a hermitian
  derivation, and the derivative of a central element is again
  central, also the second term is hermitian. This shows that 
  $h(\nabla_a\nabla_bE_p,E_q)$ is hermitian and, hence, that
  $(M,h,\gphi,\nabla)$ is a real pseudo-Riemannian calculus.
\end{proof}

\noindent Let us proceed to compute the curvature of
$(M,h,\gphi,\nabla)$. \noindent Recall that since the
pseudo-Riemannian calculus is real,
Proposition~\ref{prop:curvature.symmetries} implies that the curvature
operator has all the classical symmetries.

\begin{proposition}
  The curvature of the pseudo-Riemannian calculus $(M,h,\gphi,\nabla)$
  over $\Sthreet$ is given by
  \begin{alignat*}{3}
    &\R{1}{2}{1} = -E_2\Zsq &\quad
    &\R{1}{2}{2} = E_1\Wsq &\quad
    &\R{1}{2}{3} = 0\\
    &\R{1}{3}{1} = -E_3\Zsq&
    &\R{1}{3}{2} = 0&
    &\R{1}{3}{3} = E_1\Zsq\Wsq\\
    &\R{2}{3}{1} = 0 &
    &\R{2}{3}{2} = -E_3\Wsq & 
    &\R{2}{3}{3} = E_2\Zsq\Wsq,
  \end{alignat*}
  from which it follows that the nonzero curvature components can be
  obtained from
  \begin{align*}
    R_{1212} = \Zsq\Wsq\qquad
    R_{1313} = \paraa{\Zsq}^2\Wsq\qquad
    R_{2323} = \Zsq\paraa{\Wsq}^2.
  \end{align*}
  Moreover, the (unique) scalar curvature is given by $S=6\cdot\mid$.
\end{proposition}

\begin{proof}
  First, it is straightforward to compute $R(\d_a,\d_b)E_c$ by using
  the results in
  Proposition~\ref{prop:S3.pseudo.Riemannian.calculus}. For instance
  (recall that $[\d_a,\d_b]=0$)
  \begin{align*}
    R(\d_1,\d_3)E_3 &= \nabla_1\nabla_3E_3-\nabla_3\nabla_1E_3
    =\nabla_1\paraa{E_3(\Wsq-\Zsq)}-\nabla_3\paraa{E_1\Wsq}\\
    &=\paraa{\nabla_1E_3}(\Wsq-\Zsq)-\paraa{\nabla_3E_1}\Wsq-E_1\d_3\Wsq\\
    &=E_1\Wsq\paraa{\Wsq-\Zsq}-E_1(\Wsq)^2-E_1\paraa{-2\Wsq\Zsq}\\
    &= E_1\Zsq\Wsq.
  \end{align*}
  The components are easily computed as well; e.g.
  \begin{align*}
    R_{1212} = h\paraa{E_1,R(d_1,d_2)E_2} 
    =h(E_1,E_1\Wsq)=h(E_1,E_1)\Wsq= \Zsq\Wsq
  \end{align*}
  and
  \begin{align*}
    R_{1223} = h\paraa{E_1,R(\d_2,\d_3)E_2} = h(E_1,-E_3\Wsq) =
    -h(E_1,E_3)\Wsq = 0. 
  \end{align*}
  Computing $R_{abcd}$ for $a,b,c,d\in\{1,2,3\}$ (using the symmetries in
  Proposition~\ref{prop:curvature.symmetries} to reduce the number of
  computations that need to be performed) gives $R_{1212}$, $R_{1313}$
  and $R_{2323}$ (together with the ones obtained by symmetry from
  these) as the only nonzero components.

  Finally, let us show that there is a unique scalar curvature. The
  metric $h_{ab}$ has a pseudo-inverse $(\hhat^{ab},H)$, given by
  \begin{align*}
    (\hhat^{ab}) =
    \begin{pmatrix}
      \Wsq & 0 & 0 \\
      0 & \Zsq & 0\\
      0 & 0 & \mid
    \end{pmatrix}
    \quad\text{and}\quad H=\Zsq\Wsq.
  \end{align*}
  From the computation
  \begin{align*}
    &\qquad\hhat^{ab}R_{apbq}\hhat^{pq} = 
    \hhat^{11}R_{1p1q}\hhat^{pq}+
    \hhat^{22}R_{2p2q}\hhat^{pq}+
    \hhat^{33}R_{3p3q}\hhat^{pq}=\\
    &\hhat^{11}\paraa{R_{1212}\hhat^{22}+R_{1313}\hhat^{33}}+
    \hhat^{22}\paraa{R_{2121}\hhat^{11}+R_{2323}\hhat^{33}}+
    \hhat^{33}\paraa{R_{3131}\hhat^{11}+R_{3232}\hhat^{22}}\\
    &= 2\Wsq\Zsq\Wsq\Zsq+2\Wsq(\Zsq)^2\Wsq+2\Zsq\Zsq(\Wsq)^2= H(6\cdot\mid)H,
  \end{align*}
  one concludes that the scalar curvature with respect to
  $(\hhat^{ab},H)$ is given by $6\cdot\mid$.  Since $H$ is central, it
  follows from
  Proposition~\ref{prop:H.central.unique.scalar.curvature} that this
  is indeed the unique scalar curvature of $(M,h,\gphi,\nabla)$.
\end{proof}

\subsection{Aspects of localization on $\Sthreet$}
\label{sec:global.vs.local}

\noindent In classical geometry, the projection onto the normal space
of $S^3$ is given by the map
\begin{align*}
  \Pi(U)^i = \Pi^{ij}U^j = x^ix^jU^j, 
\end{align*}
where $x^1,\ldots,x^4$ are the embedding coordinates of $S^3$ into
$\reals^4$, satisfying
\begin{align*}
  (x^1)^2+(x^2)^2+(x^3)^2+(x^4)^2=1. 
\end{align*}
Hence, (sections of) the tangent bundle may be
identified with the projective module
\begin{align*}
  TS^3 = \P\paraa{C^\infty(S^3)^4}
\end{align*}
where $\P = \mid-\Pi$, giving $TS^3$ as a subspace of $T\reals^4$. It
is well known that $S^3$ is parallelizable, which means that $TS^3$
is a free module, and one may explicitly give a basis of (global) vector
fields as:
\begin{align*}
  v_1 = (-x^4,x^3,-x^2,x^1)\quad
  v_2 = (-x^3,-x^4,x^1,x^2)\quad
  v_3 = (-x^2,x^1,x^4,-x^3).
\end{align*}
The (global) vector fields $E_1,E_2,E_3$ 
\begin{align*}
  &E_1 = (-x^2,x^1,0,0)\qquad
  E_2 = (0,0,-x^4,x^3)\\
  &E_3 = \paraa{x^1|w|^2,x^2|w|^2,-x^3|z|^2,-x^4|z|^2}
\end{align*}
as defined in Section~\ref{sec:hopf.coordinates} are linearly independent
at every point where $|z|^2=(x^1)^2+(x^2)^2\neq 0$ and
$|w|^2=(x^3)^2+(x^4)^2\neq 0$, which can easily be seen by computing the
determinant
\begin{align*}
  \begin{vmatrix}
    -x^2 & x^1 & 0 & 0\\
    0 & 0 & -x^4 & x^3\\
    x^1|w|^2 & x^2|w|^2 & -x^3|z|^2 & -x^4|z|^2\\
    x^1 & x^2 & x^3 & x^4
  \end{vmatrix}
  =-|z|^2|w|^2,
\end{align*}
giving a condition for $E_1,E_2,E_3,\vec{n}=(x^1,x^2,x^3,x^4)$ to be
linearly independent. Thus, the vector fields $E_1,E_2,E_3$ provide a
globalization of the corresponding vector fields in the local chart
defined by the Hopf coordinates, and one may use them for
computations, keeping in mind that they do not span the tangent space
at points $(x^1,x^2,x^3,x^4)\in S^3$ where $x^1=x^2=0$ or
$x^3=x^4=0$. However, in this case, the set of points on $S^3$ which are not
covered by this chart has measure zero, which implies that certain
results, e.g. results involving integration over the manifold, is not
sensitive to the difference between $\{E_1,E_2,E_3\}$ and $\{v_1,v_2,v_3\}$. 

Returning to the noncommutative 3-sphere $\Sthreet$, it is easy to
check that since
\begin{align*}
    (X^1)^2+(X^2)^2+(X^3)^2+(X^4)^2=Z\Zs + W\Ws=\mid
\end{align*}
the above situation allows for a straightforward
generalization. Namely, for $U=e_iU^i\in\paraa{\Sthreet}^4$ one
defines the map $\P:(\Sthreet)^4\to(\Sthreet)^4$ as
\begin{align*}
  \P(U) = \sum_{i,j=1}^4 e_i\P^{ij}U^j
\end{align*}
where $\P^{ij} = \delta^{ij}\mid-X^iX^j$, and it is easy to check that
$\P^2(U)=\P(U)$. Hence, $T\Sthreet=\P\paraa{(\Sthreet)^4}$ is a
projective module in close analogy with the module of vector fields on
$S^3$.  Let us now further study the structure of $T\Sthreet$. We
start by proving the following lemma.
\begin{lemma}\label{lemma:X.lemma}
In $\Sthreet$
\begin{align}
  &X^2X^4+X^1X^3=q\paraa{X^4X^2+X^3X^1}\label{eq:X.lemma.1}\\
  &X^2X^4-X^1X^3=\qb\paraa{X^4X^2-X^3X^1}\label{eq:X.lemma.2}\\
  &X^2X^3+X^1X^4=\qb\paraa{X^3X^2+X^4X^1}\label{eq:X.lemma.3}\\
  &X^2X^3-X^1X^4=q\paraa{X^3X^2-X^4X^1}\label{eq:X.lemma.4}.
\end{align}
\end{lemma}

\begin{proof}
  The proof is done by straightforward computations; e.g.
  \begin{align*}
    X^2X^3+X^1X^4 &= \frac{1}{2i}(Z-\Zs)\frac{1}{2}(W+\Ws)
    +\frac{1}{2}(Z+\Zs)\frac{1}{2i}(W-\Ws)\\
    &=\frac{1}{2i}\paraa{ZW-\Zs\Ws}=\frac{\qb}{2i}\paraa{WZ-\Ws\Zs}\\
    &=\frac{1}{2}(W+\Ws)\frac{1}{2i}(Z-\Zs)
    +\frac{1}{2i}(W-\Ws)\frac{1}{2}(Z+\Zs)\\
    &=\qb\paraa{X^3X^2+X^4X^1},
  \end{align*}
  and the remaining computations are completely analogous. 
\end{proof}

\noindent
The next statement corresponds to the fact that $S^3$ is a
parallelizable manifold.

\begin{proposition}
  The (right) $\Sthreet$-module $T\Sthreet$ is a free module with
  basis
  \begin{align*}
    F_1 &= (-X^4,X^3,-qX^2,qX^1)\\
    F_2 &= (-X^3,-X^4,qX^1,qX^2)\\
    F_3 &= (-X^2,X^1,X^4,-X^3).
  \end{align*}
\end{proposition}

\begin{proof}
  Let us start by showing that $\Pi(F_a)=0$, which implies that
  $F_a\in T\Sthreet$. Since $\Pi^{ij}=X^iX^j$, it is enough to show
  that $X^iF_a^i=0$ for $a=1,2,3$:
  \begin{align*}
    &X^iF_1^i = -X^1X^4+X^2X^3-qX^3X^2+qX^4X^1=0\\
    &X^iF_2^i = -X^1X^3-X^2X^4+qX^3X^1+qX^4X^2=0\\
    &X^iF_3^i = -X^1X^2+X^2X^1+X^3X^4-X^4X^3=0,
  \end{align*}
  by using \eqref{eq:X.lemma.4}, \eqref{eq:X.lemma.1} in
  Lemma~\ref{lemma:X.lemma}, and the fact that
  $[X^1,X^2]=[X^3,X^4]=0$.

  Next, we show that $F_1,F_2,F_3$ generate $T\Sthreet$; it is
  sufficient to show that $\P(e_i)$ (where $\{e_i\}_{i=1}^4$ denotes
  the canonical basis of $(\Sthreet)^4$) can be written as linear
  combination of $F_1,F_2,F_3$, for $i=1,2,3,4$. In fact, one can show that
  \begin{align*}
    &\P(e_1) = \paraa{\mid-(X^1)^2,-X^2X^1,-X^3X^1,-X^4X^1}
    =-F_1X^4-F_2X^3-F_3X^2\\
    &\P(e_2) = \paraa{-X^1X^2,\mid-(X^2)^2,-X^3X^2,-X^4X^2}
    =F_1X^3-F_2X^4+F_ 3X^1\\
    &\P(e_3) = \paraa{-X^1X^3,-X^2X^3,\mid-(X^3)^2,X^4X^3}
    =-\qb F_1X^2+\qb F_ 2X^1+F_3X^4\\
    &\P(e_4) = \paraa{-X^1X^4,-X^2X^4,-X^3X^4,\mid-(X^4)^2}
    =\qb F_1X^1 + \qb F_ 2X^2-F_ 3X^3.
  \end{align*}
  For instance,
  \begin{align*}
    -F_1X^4-F_2X^3-F_3X^2&
    =\paraa{(X^2)^2+(X^3)^2+(X^4)^2,-X^3X^4+X^4X^3-X^1X^2,\\
    &qX^2X^4-qX^1X^3-X^4X^2,-qX^1X^4-qX^2X^3+X^3X^2}\\
    &= \paraa{\mid-(X^1)^2,-X^2X^1,-X^3X^1,-X^4X^1}=\P(e_1),
  \end{align*}
  by using \eqref{eq:X.lemma.2}, \eqref{eq:X.lemma.3} (in the third
  and fourth component, respectively) and the fact
  that $[X^1,X^2]=[X^3,X^4]=0$. Finally, let us show that
  $F_1,F_2,F_3$ are free generators. For $a,b,c\in\Sthreet$, we assume that
  \begin{align*}
    F_1a + F_2b + F_3c = 0,
  \end{align*}
  which is equivalent to
  \begin{align*}
    \begin{cases}
      -X^4a-X^3b-X^2c = 0\\
      X^3a-X^4b+X^1c = 0\\
      -qX^2a+qX^1b+X^4c = 0\\
      qX^1a+qX^2b-X^3c = 0.
    \end{cases}
  \end{align*}
  Multiplying these equations (from the left) by $-X^2$, $X^1$, $X^4$
  and $-X^3$, respectively, and summing them yields $c=0$, by using 
  \eqref{eq:X.lemma.1} and \eqref{eq:X.lemma.4}. Setting $c=0$ in the
  above equations gives
  \begin{align*}
    &X^4a+X^3b = 0\qquad X^3a-X^4b = 0\\
    -&X^2a+X^1b=0\qquad X^1a+X^2b = 0,
  \end{align*}
  which implies that
  \begin{align*}
    (X^4)^2a = -X^4X^3b\qquad (X^3)^2a=X^3X^4b\\
    (X^2)^2a=X^2X^1b\qquad (X^1)^2a = -X^1X^2b.
  \end{align*}
  Summing these equations gives $a=0$, which then (via a similar argument) implies that
  $b=0$. This shows that $F_1,F_2,F_3$ are linearly independent.
\end{proof}

\noindent
It is easy to check that the elements $E_1,E_2,E_3$, as defined in
\eqref{Sthree.Evec.def}, fulfill $\P(E_a)=E_a$ for $a=1,2,3$, implying
that they are elements of $T\Sthreet$. Hence, the module $M$, of the
pseudo-Riemannian calculus for $\Sthreet$, is a submodule of
$T\Sthreet$, providing a noncommutative analogue of the globalization of the local
vector fields in the Hopf coordinates as described in the beginning of
the section.

As is well known, every projective module comes equipped with a
canonical affine connection; namely, the module $(\Sthreet)^4$ has an
affine connection, given by
\begin{align*}
  \nablab_dV = e_id(V^i) 
\end{align*}
where $V=e_iV^i\in(\Sthreet)^4$ and $d\in\Der(\Sthreet)$, and it
follows that 
\begin{align*}
  \nablah_dV = \P(\nablab_dV)
\end{align*}
is an affine connection on $T\Sthreet$. Since we have argued in
  analogy with differential geometry, where $M$ is a sub-module of
  $TS^3$ and the connection on $M$ is merely the restriction of the
  connection on $TS^3$, it is natural to ask if the connection
  $\nablah$ (restricted to $M$) coincides with $\nabla$ (as given by
  the pseudo-Riemannian calculus over $M$).

\begin{proposition}
  Let $(M,h,\gphi,\nabla)$ be the pseudo-Riemannian calculus over
  $\Sthreet$ introduced in Section
  \ref{sec:pseudo.Riemannian.calculus.Sthreet}.  The affine connection
  $\nablah_dU=\P(\nablab_d U)$, restricted to $M\subseteq T\Sthreet$, coincides
  with $\nabla$; i.e, $\nablah_d U = \nabla_dU$ for
  $d\in\g$ and $U\in M$.
\end{proposition}

\begin{proof}
  The proof is easily done by a straightforward computation, where one
  computes $\nablah_aE_b$ for $a,b=1,2,3$, and compares it with the
  result in Proposition~\ref{prop:S3.pseudo.Riemannian.calculus}. For
  instance,
  \begin{align*}
    \nablah_1E_1 &= \P\paraa{(-\d_1X^2,\d_1X^1,0,0)}
    =\P\paraa{(-X^1,-X^2,0,0)} \\
    &= (-X^1,-X^2,0,0)-(X^1,X^2,X^3,X^4)\paraa{-(X^1)^2-(X^2)^2}\\
    &=\paraa{X^1(\Zsq-\mid),X^2(\Zsq-\mid),X^3\Zsq,X^4\Zsq}\\
    &=\paraa{-X^1\Wsq,-X^2\Wsq,X^3\Zsq,X^4\Zsq} = -E_3,
  \end{align*}
  which coincides with $\nabla_1E_1$.
\end{proof}

\noindent 
In order to take the analogy with localization one step further, let us
introduce a localized algebra $\Sthreetloc$ constructed by formally
adjoining the inverses of $\Zsq$ and $\Wsq$ to the algebra
$\Sthreet$. More precisely, the multiplicative set $S$ generated by
$\Zsq,\Wsq,\mid$ trivially satisfies the (right and left) Ore condition
(since it consists of central elements) and the fact that $\Zsq,\Wsq$
are regular elements (cf. Proposition~\ref{prop:Sthree.zero.div}) implies
that the Ore localization at $S$ exists (see e.g. \cite{c:skewFields}). If we consider $T\Sthreet$
and $M$ as (right) $\Sthreetloc$-modules, they coincide, which we show
by explicitly finding a relation between the two sets of generators.

\begin{proposition}
  Consider the following elements of $(\Sthreetloc)^4$:
  \begin{alignat*}{2}
    &F_1= (-X^4,X^3,-qX^2,qX^1) & &E_1 = (-X^2,X^1,0,0)\\
    &F_2= (-X^3,-X^4,qX^1,qX^2) & &E_2 = (0,0,-X^4,X^3)\\
    &F_3=(-X^2,X^1,X^4,-X^3) &\qquad &E_3 = (X^1\Wsq, X^2\Wsq, -X^3\Zsq, -X^4\Zsq).
  \end{alignat*}
  Then it holds that
  \begin{align*}
    F_1 &= E_1|Z|^{-2}\paraa{X^1X^3+X^2X^4}
    +E_2|W|^{-2}\paraa{X^1X^3+X^2X^4}\\
    &\qquad\quad+E_3|Z|^{-2}|W|^{-2}\paraa{X^2X^3-X^1X^4}\\
    F_2 &=E_1|Z|^{-2}\paraa{X^2X^3-X^1X^4}
    +E_2|W|^{-2}\paraa{X^2X^3-X^1X^4}\\
    &\qquad\quad-E_3|Z|^{-2}|W|^{-2}\paraa{X^1X^3+X^2X^4}\\
    F_3 &=E_1-E_2.
  \end{align*}
\end{proposition}

\begin{proof}
  Let us show that $F_1$ can be written as a linear combination of
  $E_1,E_2,E_3$, as given in the statement. Namely, introducing $W^i$ through
  \begin{align*}
    e_iW^i = E_1&|Z|^{-2}\paraa{X^1X^3+X^2X^4}
    +E_2|W|^{-2}\paraa{X^1X^3+X^2X^4}\\
    &+E_3|Z|^{-2}|W|^{-2}\paraa{X^2X^3-X^1X^4},
  \end{align*}
  gives
  \begin{align*}
    W^1 &= -X^2|Z|^{-2}(X^1X^3+X^2X^4)+X^1|Z|^{-2}(X^2X^3-X^1X^4)\\
    W^2 &= X^1|Z|^{-2}(X^1X^3+X^2X^4)+X^2|Z|^{-2}(X^2X^3-X^1X^4)\\
    W^3 &= -X^4|W|^{-2}(X^1X^3+X^2X^4)-X^3|W|^{-2}(X^2X^3-X^1X^4)\\
    W^4 &= X^3|W|^{-2}(X^1X^3+X^2X^4)-X^4|W|^{-2}(X^2X^3-X^1X^4).
  \end{align*}
  Using the fact that $[X^1,X^2]=0$ (in $W^1,W^2$), together with
  \eqref{eq:X.lemma.1} and \eqref{eq:X.lemma.4} (in $W^3,W^4$), yields
  \begin{align*}
    W^1 &= -|Z|^{-2}\paraa{(X^2)^2+(X^1)^2}X^4=-|Z|^{-2}\Zsq X^4 = -X^4\\
    W^2 &= |Z|^{-2}\paraa{(X^1)^2+(X^2)^2}X^3=|Z|^{-2}\Zsq X^3=X^3\\
    W^3 &= -q|W|^{-2}\paraa{(X^4)^2+(X^3)^2}X^2=-q|W|^{-2}\Wsq X^2=-qX^2\\
    W^4 &= q|W|^{-2}\paraa{(X^4)^2+(X^3)^2}X^1=q|W|^{-2}\Wsq X^1=qX^1,
  \end{align*}
  which shows that $e_iW^i=F_1$.
\end{proof}

\noindent
Finally, we note that the metric 
\begin{align*}
  \paraa{h_{ab}}=
  \begin{pmatrix}
    \Zsq & 0 & 0 \\
    0 & \Wsq & 0 \\
    0 & 0 & \Zsq\Wsq
  \end{pmatrix}
\end{align*}
is invertible in $\Sthreetloc$, giving a local calculus in almost complete
analogy with differential geometry.

\section*{Acknowledgment}

\noindent We would like to thank L. Dabrowski, M. Khalkhali, G. Landi,
and A. Sitarz for interesting and useful discussions during the HIM
trimester program ``Noncommutative Geometry and its Applications'' in
the fall of 2014, as well as the University of Western Ontario
  for hospitality. Furthermore, J.A. is supported by the Swedish
Research Council.

\bibliographystyle{alpha}
\bibliography{sphere_curvature}  

\newcommand{\etalchar}[1]{$^{#1}$}
\def\polhk#1{\setbox0=\hbox{#1}{\ooalign{\hidewidth
  \lower1.5ex\hbox{`}\hidewidth\crcr\unhbox0}}}
\begin{thebibliography}{DVMMM96}

\bibitem[ABH{\etalchar{+}}09]{abhhs:noncommutative}
J.~Arnlind, M.~Bordemann, L.~Hofer, J.~Hoppe, and H.~Shimada.
\newblock Noncommutative {R}iemann surfaces by embeddings in {$\Bbb R\sp 3$}.
\newblock {\em Comm. Math. Phys.}, 288(2):403--429, 2009.

\bibitem[AC10]{ac:ncgravitysolutions}
P.~Aschieri and L.~Castellani.
\newblock Noncommutative gravity solutions.
\newblock {\em J. Geom. Phys.}, 60(3):375--393, 2010.

\bibitem[AH14]{ah:pseudoRiemannian}
J.~Arnlind and G.~Huisken.
\newblock Pseudo-{R}iemannian geometry in terms of multi-linear brackets.
\newblock {\em Lett. Math. Phys.}, 104(12):1507--1521, 2014.

\bibitem[AHH12]{ahh:multilinear}
J.~Arnlind, J.~Hoppe, and G.~Huisken.
\newblock Multi-linear formulation of differential geometry and matrix
  regularizations.
\newblock {\em J. Differential Geom.}, 91(1):1--39, 2012.

\bibitem[Arn14]{a:curvatureGeometric}
J.~Arnlind.
\newblock Curvature and geometric modules of noncommutative spheres and tori.
\newblock {\em J. Math. Phys.}, 55:041705, 2014.

\bibitem[BM11]{bm:starCompatibleConnections}
E.~J. Beggs and S.~Majid.
\newblock {$*$}-compatible connections in noncommutative {R}iemannian geometry.
\newblock {\em J. Geom. Phys.}, 61(1):95--124, 2011.

\bibitem[CFF93]{cff:gravityncgeometry}
A.~H. Chamseddine, G.~Felder, and J.~Fr{\"o}hlich.
\newblock Gravity in noncommutative geometry.
\newblock {\em Comm. Math. Phys.}, 155(1):205--217, 1993.

\bibitem[CM14]{cm:modularCurvature}
A.~Connes and H.~Moscovici.
\newblock Modular curvature for noncommutative two-tori.
\newblock {\em J. Amer. Math. Soc.}, 27(3):639--684, 2014.

\bibitem[Coh95]{c:skewFields}
P.~M. Cohn.
\newblock {\em Skew fields}, volume~57 of {\em Encyclopedia of Mathematics and
  its Applications}.
\newblock Cambridge University Press, Cambridge, 1995.
\newblock Theory of general division rings.

\bibitem[Con80]{c:cstaralgebre}
A.~Connes.
\newblock {$C^{\ast} $} alg\`ebres et g\'eom\'etrie diff\'erentielle.
\newblock {\em C. R. Acad. Sci. Paris S\'er. A-B}, 290(13):A599--A604, 1980.

\bibitem[Con94]{c:ncgbook}
A.~Connes.
\newblock {\em Noncommutative geometry}.
\newblock Academic Press, Inc., San Diego, CA, 1994.

\bibitem[CT11]{ct:gaussBonnet}
A.~Connes and P.~Tretkoff.
\newblock The {G}auss-{B}onnet theorem for the noncommutative two torus.
\newblock In {\em Noncommutative geometry, arithmetic, and related topics},
  pages 141--158. Johns Hopkins Univ. Press, Baltimore, MD, 2011.

\bibitem[DV88]{dv:calculDifferentiel}
M.~Dubois-Violette.
\newblock D\'erivations et calcul diff\'erentiel non commutatif.
\newblock {\em C. R. Acad. Sci. Paris S\'er. I Math.}, 307(8):403--408, 1988.

\bibitem[DVMMM96]{dvmmm:onCurvature}
M.~Dubois-Violette, J.~Madore, T.~Masson, and J.~Mourad.
\newblock On curvature in noncommutative geometry.
\newblock {\em J. Math. Phys.}, 37(8):4089--4102, 1996.

\bibitem[FK12]{fk:gaussBonnet}
F.~Fathizadeh and M.~Khalkhali.
\newblock The {G}auss-{B}onnet theorem for noncommutative two tori with a
  general conformal structure.
\newblock {\em J. Noncommut. Geom.}, 6(3):457--480, 2012.

\bibitem[FK13]{fk:scalarCurvature}
F.~Fathizadeh and M.~Khalkhali.
\newblock Scalar curvature for the noncommutative two torus.
\newblock {\em J. Noncommut. Geom.}, 7(4):1145--1183, 2013.

\bibitem[FK15]{fk:curvatureFourTori}
F.~Fathizadeh and M.~Khalkhali.
\newblock Scalar curvature for noncommutative four-tori.
\newblock {\em J. Noncommut. Geom.}, 9(2):473--503, 2015.

\bibitem[LM15]{lm:modularCurvature}
M~Lesch and H.~Moscovici.
\newblock Modular curvature and {M}orita equivalence.
\newblock \texttt{arXiv:1505.00964}, 2015.

\bibitem[Mat91a]{m:ncspheres}
K.~Matsumoto.
\newblock Noncommutative three-dimensional spheres.
\newblock {\em Japan. J. Math. (N.S.)}, 17(2):333--356, 1991.

\bibitem[Mat91b]{m:ncspheresII}
K.~Matsumoto.
\newblock Noncommutative three-dimensional spheres. {II}. {N}oncommutative
  {H}opf fibering.
\newblock {\em Yokohama Math. J.}, 38(2):103--111, 1991.

\bibitem[Ros13]{r:leviCivita}
J.~Rosenberg.
\newblock {L}evi-{C}ivita's theorem for noncommutative tori.
\newblock {\em SIGMA}, 9:071, 2013.

\end{thebibliography}

\end{document}